\documentclass[11pt]{amsart}

\usepackage[all]{xy}
\usepackage{amsmath, amsfonts, amscd, latexsym, amsthm, amssymb}
\usepackage{hyperref}
\usepackage{tikz}

\usepackage[margin=1.2in]{geometry}

\def \c{\mathbb{C}}
\def \z{\mathbb{Z}}
\def \r{\mathbb{R}}
\def \n{\mathbb{N}}
\def \p{\mathbb{P}}
\def \q{\mathbb{Q}}
\def \t{\mathfrak{t}}

\def \V{\mathcal{V}}
\def \U{\mathcal{U}}
\def \C{\mathcal{C}}
\def \B{\mathcal{B}}
\def \A{\mathcal{A}}
\def \F{\mathcal{F}}
\def \SL{\textup{SL}}
\def \GL{\textup{GL}}
\def \Sp{\textup{Sp}}
\def \SO{\textup{SO}}
\def \S{{\underline{w_0}}}
\def \w{{\underline{w}}}

\def \X{\tilde{X}}
\def \Y{\tilde{Y}}

\def \dim{\textup{dim}}

\def \Vol{\textup{Vol}}

\def \wt{\textup{wt}}
\def \Lie{\textup{Lie}}
\def \Proj{\textup{Proj}}

\def \gr{\textup{gr}}
\def \max{\textup{max}}

\def \conv{\textup{conv}}

\theoremstyle{plain}
\newtheorem{THM}{Theorem}
\newtheorem{COR}{Corollary}

\theoremstyle{plain}
\newtheorem{Th}{Theorem}[section]
\newtheorem{Lem}[Th]{Lemma}
\newtheorem{Prop}[Th]{Proposition}
\newtheorem{Cor}[Th]{Corollary}

\theoremstyle{definition}
\newtheorem{Ex}[Th]{Example}
\newtheorem{Def}[Th]{Definition}
\newtheorem{Rem}[Th]{Remark}

\begin{document}
\title{Crystal bases and Newton-Okounkov bodies}
\author{Kiumars Kaveh}
\date{\today}
\dedicatory{Dedicated to my teacher Prof. Siavash Shahshahani}

\thanks{The author is partially supported by a
Simons Foundation Collaboration Grant for Mathematicians and a National Science Foundation Grant.}

\begin{abstract}
Let $G$ be a connected reductive algebraic group.
We prove that the string parametrization of a crystal basis for a finite dimensional
irreducible representation of $G$ { extends to a natural valuation on the field of rational
functions} on the flag variety $G/B$, which is a highest term valuation corresponding to a 
coordinate system on a Bott-Samelson variety. This shows that the string polytopes associated to irreducible representations,
can be realized as Newton-Okounkov bodies for the flag variety. This is closely related to an earlier
result of A. Okounkov for the Gelfand-Cetlin polytopes of the symplectic group \cite{Ok-NP}.
As a corollary we { recover a multiplicativity property} of the canonical basis due to P. Caldero.
We generalize the results to spherical varieties. From these the existence of SAGBI bases for the homogeneous
coordinate rings of flag and spherical varieties, as well as their toric degenerations follow recovering results in 
\cite{Caldero, AB, Kaveh-Michigan}.\\
\end{abstract}

\maketitle

\noindent{\it Key words:} string polytope, Gelfand-Cetlin polytope,
crystal basis, string parametrization,
flag variety, Bott-Samelson variety, spherical variety, Newton-Okounkov body, 
SAGBI basis, toric degeneration.\\

\noindent{\it Subject Classification: } 14M15, 05E10, 14M27.

\setcounter{tocdepth}{1}
\tableofcontents

\section*{Introduction} \label{sec-intro}
Let $G$ be a connected reductive algebraic group over $\c$. In this paper we make a connection between
the theory of crystal bases for irreducible representations of $G$ and their string parameterizations on one hand, 
and the geometry of the flag variety of $G$ in connection with Gr\"obner theory and highest term of polynomials on the other hand.
{More precisely, we show that the
string parametrization of a (dual) crystal basis (due to Littelmann \cite{Littelmann} and Berenstein-Zelevinsky \cite{BZ-CB})
extends to a natural geometric valuation on the field of rational functions $\c(G/B)$, 
constructed out of a coordinate system on a Bott-Samelson variety, where we regard the elements of the 
irreducible representation as polynomials on the open cell in $G/B$ and hence rational functions on $G/B$.} 
This interpretation of the string parametrization shows that the string polytopes associated to irreducible representations
can be realized as Newton-Okounkov bodies for the flag variety of $G$. The notion of a Newton-Okounkov body is a far generalization of the notion of the Newton polytope of a toric variety (see \cite{Ok-BM, Ok-LC, KKh-affine, LM, KKh-NO}). We believe that this opens new doors to study the fundamental notion of a crystal basis/canonical basis in representation theory 
and we expect it to make some properties of the crystal bases more 
transparent. As an example, we readily deduce a multiplicativity property of the dual canonical basis 
for { the algebra of $U$-invariants} $\c[G]^U$ due to P. Caldero (\cite{Caldero}). From { this multiplicativity property} one then obtains toric degenerations of flag varieties and
spherical varieties recovering results in \cite{Caldero, Kaveh-Michigan, AB}.

The motivation for the main result of the paper goes back to a result of A. Okounkov who showed that
when $G = \Sp(2n, \c)$, the set of integral points in the Gelfand-Cetlin polytope of an irreducible
representation of $G$ can be identified with the collection of lowest terms of elements of this
representation regarded as polynomials on the open cell in the flag variety, and with respect to a natural coordinate system (\cite{Ok-NP}).

Below we briefly explain the key ingredients of the main result namely (1) crystal bases and their string parametrization, and (2) valuations and 
Newton-Okounkov bodies. 

Let $V_\lambda$ be a finite dimensional irreducible representation of $G$ with highest weight $\lambda$.
There is a remarkable basis for $V_\lambda$, consisting of weight vectors, called the {\it canonical basis/crystal basis}, which 
combinatorially encodes the action of $\Lie(G)$
(\cite{Lusztig}, also Section \ref{sec-crystal}).\footnote{The paper \cite{Lusztig} gives the first construction of the canonical basis (of the quantum enveloping algebra) 
for general reductive groups. We are interested in the specialization at $q = 1$ of the Lusztig canonical basis. The Kashiwara crystal bases are $q = 0$ case of 
the canonical basis and appeared in \cite{Kashiwara-crystal} in which their existence was proved for classical groups.}

Throughout the paper we denote the crystal basis for $V_\lambda$ by $\B_\lambda$ and the corresponding dual basis for $V_\lambda^*$ by $\B_\lambda^*$ which we will
refer to as the {\it dual crystal basis}.\footnote{In representation theory literature, a crystal basis usually refers to the specialization at $q = 0$ of the canonical basis for the quantum 
enveloping algebra of the group. Here we use this term for the corresponding basis for an irreducible representation of G itself. In this sense our use of the term crystal basis deviates from the literature.} 
Crystal bases play a fundamental role in
the representation theory of $G$. There is a nice parametrization of the elements of a (dual) crystal basis, called the {\it string parametrization}, by the set of integral points in a certain polytope in $\r^N$, where $N = \dim(G/B)$ is the number of positive roots (\cite{Littelmann}, \cite{BZ-CB}, see also Section \ref{sec-string-para}). This parametrization depends on a combinatorial choice namely a reduced decomposition
for the longest element $w_0$ in the Weyl group. That is, an $N$-tuple of simple reflections
$\S = (s_{\alpha_{i_1}}, \ldots, s_{\alpha_{i_N}})$ with
$$w_0 = s_{\alpha_{i_1}} \cdots s_{\alpha_{i_N}}.$$
Here $\alpha_1, \ldots, \alpha_r$ are { the simple roots} and $s_\alpha$ denotes the simple reflection corresponding to a simple root $\alpha$.
{ We denote by $\iota_\S: \B^*_\lambda \to \z^N$ the corresponding string parametrization for the dual crystal basis $\B^*_\lambda$}. 
The polytope associated to $V_\lambda$ and a reduced decomposition $\S$ is called a {\it string polytope} and we
denote it by $\Delta_\S(\lambda)$. The string polytopes are generalizations of the well-known Gelfand-Cetlin polytopes of representations of $\GL(n, \c)$ (\cite{GC}).

A main property of a string polytope is that the number of integral points in the polytope $\Delta_\S(\lambda)$ is equal to
$\dim(V_\lambda)$. Any dominant weight $\lambda$ gives a $G$-linearized line bundle $L_\lambda$ on $X$ and one knows that the space of 
sections $H^0(X, L_\lambda)$ is isomorphic to the dual representation $V_\lambda^*$. It follows that 
{ the degree of the line bundle} $L_\lambda$ is given by $N!$ times the
volume of $\Delta_\S(\lambda)$ (see Corollary \ref{cor-deg-G/P}).

On the other hand following the pioneering works of A. Okounkov in \cite{Ok-BM} and  \cite{Ok-LC}, the author and A. G. Khovanskii 
(\cite{KKh-NO, KKh-affine}), as well as Lazarsfeld and Mustata (\cite{LM}), developed a theory of convex bodies associated to
linear series on algebraic varieties. Let $X$ be a $d$-dimensional projective algebraic variety with { field of rational functions} $\c(X)$.
Let $L$ be a very ample line bundle on $X$ and let $R = R(L)$ denote the corresponding
ring of sections or homogeneous coordinate ring (see Section \ref{sec-Newton-Ok-bodies}).  
Let $v: \c(X) \setminus \{0\} \to \z^d$ be a valuation with
one-dimensional leaves (see Definition~\ref{def-valuation}) on the space
of rational functions $\c(X)$, and 
let $S = S(R, v)$ denote the associated value semigroup in
$\n \times \z^d$ (see \eqref{equ-def-S} in Section \ref{sec-Newton-Ok-bodies}).
We also let $\Delta = \Delta(R, v) \subset \r^d$ denote the {\it Newton-Okounkov body}
corresponding to $R$ and $v$ (Definition~\ref{def-Newton-Okounkov-body}). 
The convex body $\Delta$ has the property that 
{ the degree of} $L$ is given by $d!$ times the
volume of $\Delta(R(L))$ (Theorem \ref{th-asymp-Hilbert-Okounkov-body}, see also \cite{KKh-affine, LM, KKh-NO}).

Given a smooth point $p$ on $X$ and a regular system of parameters $u_1, \ldots, u_d$
in a neighborhood of $p$ one can define a valuation $v$ on $\c(X)$ with values in $\z^d$
as follows: Fix a total ordering on $\z^d$ respecting addition. Let $f$ be a polynomial in $u_1, \ldots, u_d$. Let
$c_k u_1^{k_1} \cdots {u_d}^{k_d}$ be the term in $f$ with the smallest exponent $k=(k_1, \ldots, k_d)$. Define
$v(f) = (k_1, \ldots, k_d)$. Also for a rational function $h = f/g$ where $f,g$ are {polynomials in the $u_i$}, 
define $v(h) = v(f) - v(g)$. One verifies that $v$ is a valuation on $\c(X)$ with values in $\z^d$, { which we 
call} the {\it lowest term valuation} with respect to the parameters 
$u_i$ (and the order on $\z^d$). Similarly, given a polynomial $f$ in the $u_i$, one can take the term $c_\ell u_1^{\ell_1} \cdots u_d^{\ell_d}$ in $f$  with the largest exponent $\ell = (\ell_1, \ldots, \ell_d)$. Then $v(f) = (-\ell_1, 
\ldots, -\ell_d)$ defines a valuation on $\c(X)$. We call it the {\it highest term valuation} with respect to the parameters $u_i$ (and the 
order on $\z^d$). 

Now let $X= G/B$ be the flag variety of $G$. Let $X_w \subset X$ denote the Schubert variety corresponding
to a Weyl group element $w$. 
A reduced decomposition $\S = (\alpha_{i_1}, \ldots, \alpha_{i_N})$,
$w_0 = s_{\alpha_{i_1}} \cdots s_{\alpha_{i_N}},$ { gives rise to a sequence of  Schubert varieties} in $G/B$:
$$\{o\} = X_{w_N} \subset \cdots \subset X_{w_0} = X,$$ 
where $w_k = s_{\alpha_{i_{k+1}}} \cdots s_{\alpha_{i_N}}$, $X_{w_k}$ is the Schubert variety corresponding to the Weyl group element $w_k$, 
and $o = eB$ is the unique $B$-fixed point in $X$.
Moreover, the reduced decomposition $\S$ gives a birational model $\tilde{X}_\S$ of $X$ called a Bott-Samelson variety and a 
coordinate system $\{t_1, \ldots, t_N\}$ at a point above $o$. In this coordinate system, for each $k$ the preimage of $X_{w_k}$ is 
given by $t_1 = \cdots = t_k = 0$ (see Section \ref{sec-Bott-Samelson}).

Let $v_\S$ be the highest term valuation on $\c(\tilde{X}_\S) \cong \c(X)$ associated to the coordinate system $t_1, \ldots, t_N$.
The main result of the paper is the following (Theorem \ref{th-main}):
\begin{THM} \label{THM-main}
The string parametrization for a dual crystal basis in $V_\lambda^* \cong H^0(X, L_\lambda)$
{ is the restriction of the valuation} $v_\S$. More precisely, let $\tau_\lambda$ denote a lowest weight vector in $H^0(X, L_\lambda)$. { For any 
$\sigma \in \B_\lambda^*$} write $\sigma = f_\sigma \tau_\lambda$. Then we have $\iota_\S(\sigma)=- v_\S(f_\sigma)$ where $\iota_\S$ denotes the 
string parametrization corresponding to the reduced decomposition $\S$.
It follows that the string polytope $\Delta_\S(\lambda)$ coincides with the
Newton-Okounkov body of the algebra of sections $R(L_\lambda)$ and the valuation
$v_\S$. 
\end{THM}

{ Theorem \ref{THM-main} can be generalized in a straightforward way to when $X=G/B$ and $V_\lambda^*$ are replaced with a Schubert variety $X_w$ and
a dual Demazure module $V_\lambda(w)^*$ respectively (see Remark \ref{rem-Schubert-main}).}

Let $X$ be a $d$-dimensional variety.
Given a sequence of subvarieties $$Y_{\bullet}: Y_d \subset Y_{d-1} \subset \cdots \subset Y_0 = X,$$ where each $Y_k$ has codimension $k$ and is 
non-singular along $Y_{k+1}$, one can construct a valuation $v_{Y_\bullet}$ on $\c(X)$ with one-dimensional leaves (see Example \ref{ex-Parshin-valuation}). 
It is related to and generalizes the construction of the lowest term valuation explained above. The highest weight valuation can be thought of as corresponding to a sequence of subvarieties at infinity (with respect to the chosen regular system of parameters). 
{ The valuations of the form $v_{Y_\bullet}$ constitute a large class of valuations on $\c(X)$ and one often is interested in Newton-Okounkov bodies with respect to such valuations (see for example \cite{LM} and \cite{KLM}). It is a natural question whether the highest term valuation $v_\S$ can be realized as a valuation corresponding to a sequence 
of subvarieties $Y_\bullet$ on the flag variety.
One candidate for such a sequence of subvarieties on $X = G/B$ can be defined as follows:} for each $k$ let $Y_k = w_0w_k^{-1}X_k$ be the Schubert variety of $w_k$ translated by $w_0w_k^{-1}$. In an earlier version of the paper it was claimed that the highest term valuation $v_{\S}$ coincides with the valuation $v_{Y_\bullet}$ corresponding to this sequence of translated Schubert varieties. After a discussion with Dave Anderson, the author realized that this is not true in general. 
{ Also later a calculation of Valentina Kiritchenko showed that for $G=\Sp(4, \c)$ the string polytope $\Delta_\S(\lambda)$ and the Newton-Okounkov polytope 
$\Delta(R(L_\lambda), v_{Y_\bullet})$ corresponding to the same reduced decomposition $\S$ might not even be combinatorially the same (see Remark \ref{rem-Anderson}).}

It is well-known that the algebra $\c[G]^U$ of unipotent invariants on $G$ decomposes as:
$$\c[G]^U = \bigoplus_{\lambda} V_\lambda^*.$$
There is a natural basis $\B^*$ for this algebra such that
$\B^*_\lambda = \B^* \cap V_\lambda^*$ { is the dual crystal basis for $V_\lambda^*$}.
More specifically, for each $\lambda$, $\B_\lambda^*$ is the dual of the basis $\B_\lambda$ for $V_\lambda$ 
consisting of the nonzero $bv_\lambda$, where $b$ lies in the specialization
at $q = 1$ of the Lusztig canonical basis (\cite{Lusztig}). 
In this paper we call $\B^*$ the {\it dual canonical basis} of the algebra $\c[G]^U$.
We observe that from the defining properties of a valuation and Theorem \ref{THM-main}
the following multiplicativity result (due to Caldero) readily follows
(\cite[Section 2]{Caldero} and Corollary \ref{cor-multi-prop-canonical-basis}):

\begin{COR} \label{COR-multi-prop}
Let $\lambda, \mu$ be two dominant weights. Take $b'^*, b''^* \in \B^*$ which lie in $\B_\lambda^*$ and $\B_\mu^*$ respectively. Then
the product $b'^*b''^* \in V^*_{\lambda + \mu} \subset \c[G]^U$ can be uniquely written as
$$b'^* b''^* = cb^* + \sum_j c_j b_j^*,$$ where $0 \neq c \in \c$ and
$b^*$, $b_j^*$ are in $\B^*_{\lambda+\mu}$ with $\iota_\S(b^*) = \iota_\S(b'^*) +
\iota_\S(b''^*)$, and $\iota_\S(b_j^*) < \iota_\S(b'^*) + \iota_\S(b''^*)$ whenever $c_j \neq 0$ (with respect to the 
lexicographic order). Here $\iota_\S: \B^*_\lambda \to \Delta_\S(\lambda) \cap \z^N$ is the string parametrization map 
(Section \ref{sec-string-para}).
\end{COR}

Let $X$ be a projective $G$-variety together with a $G$-linearized very ample line bundle $L$.
Consider the ring of sections $R(L) = \bigoplus_{k \geq 0} H^0(X, L^{\otimes k}).$
It is a graded $G$-algebra. To $(X, L)$ one associates a polytope $\Delta_{mom}(X, L)$ in { the positive Weyl chamber} $\Lambda_\r^+$, 
called the {\it moment polytope}, which encodes information about the
asymptotic behavior of irreducible representations appearing in $R(L)$ (see Section \ref{sec-spherical}).
When $X$ is smooth $\Delta_{mom}(X, L)$ can be identified with the Kirwan polytope (moment polytope) of $X$ regarded as a Hamiltonian space
for the action of a maximal compact subgroup $K$ of $G$.

Recall that { a normal $G$-variety} $X$ is {\it spherical} if a Borel subgroup (and hence any Borel subgroup) has a
dense orbit. A quasi-projective normal $G$-variety $X$ is spherical if for any $G$-linearized line bundle $L$, the
space of sections $H^0(X, L)$ is a multiplicity-free $G$-module. By the Bruhat decomposition flag varieties are spherical.
Let $X$ be a projective spherical $G$-variety. 
Fix a reduced decomposition $\S = (\alpha_{i_1}, \ldots, \alpha_{i_N})$.
In \cite{Ok-spherical} (motivated by a question of A. G. Khovanskii) and in \cite{AB}, the authors associate a larger polytope
$\Delta_\S(X, L)$ to $(X, L)$ defined by:
$$\Delta_\S(X, L) = \bigcup_{\lambda \in \Delta_{mom}(X, L)} \{(\lambda, x) \mid x \in \Delta_\S(\lambda) \}.$$
$\Delta_\S(X, L)$ has the property that its volume gives { the degree of the line bundle} $L$. 
This resembles (and generalizes) the Newton polytope of a
toric variety (see also \cite{Kaz, Valentina, Kaveh-JLT, KKh-Arnold}). In Section \ref{sec-spherical} we show { the following}
(Corollary \ref{cor-moment-polytope-Newton-Ok-polytope}
and Corollary \ref{cor-Newton-Okounkov-polytope-spherical}):
\begin{THM} \label{THM-1}
{ Let $X$ be a projective spherical $G$-variety and $L$ a very ample $G$-linearized line bundle on $X$.}
Then both polytopes $\Delta_{mom}(X, L)$ and $\Delta_\S(X, L)$
can be realized as Newton-Okounkov bodies for the ring of sections $R(L)$ with respect to
certain natural valuations.
\end{THM}

Let $O$ be the open $G$-orbit in $X$ and let us assume that there is a so-called wonderful compactification $Y$ of $O$ (see \cite[Definition 30.1]{Timashev}).
A wonderful compactification for $O$ exists if for example $O = G/H$ with $H = N_G(H)$. Then $Y$ is birationally isomorphic to $X$. We expect that: the valuation corresponding to $\Delta_{mom}(X, L)$ in Theorem \ref{THM-1} coincides with the valuation associated to a sequence of $G$-invariant subvarieties in $Y$.

Fix a total order $<$ on $\z^n$ respecting addition. Let $A$ be a subalgebra
of the polynomial ring $\c[x_1, \ldots, x_n]$. A subset $f_1, \ldots, f_r \in A$ is called a
{\it SAGBI basis} for $A$ (Subalgebra Analogue of Gr\"{o}bner Basis for Ideals) if
the set of initial terms of the $f_i$ (with respect to $<$) generates the semigroup of initial terms in
$A$ (in particular this semigroup is finitely generated). Given a SAGBI basis for $A$ one can represent each
$f \in A$ as a polynomial in the $f_i$ via a simple classical algorithm (known as the {\it subduction algorithm}).
There are not many examples of subalgebras known to have a SAGBI basis. It is an important unsolved problem to determine which subalgebras have a SAGBI basis.

We generalize the notion of SAGBI basis to the context of valuations on graded algebras in Section \ref{sec-SAGBI}
(Definition \ref{def-SAGBI-valuation}). \footnote{Around the same time the first version of this paper was posted in arXiv.org the paper \cite{Manon} appeared in which independently the same notion as in Definition \ref{def-SAGBI-valuation} is introduced under the name {\it subductive set}.}
In Section \ref{sec-spherical} we see that (Corollary \ref{cor-SAGBI-basis-spherical}):
\begin{COR}
The ring of sections of any $G$-linearized very ample line bundle $L$ on a
projective spherical variety has a SAGBI basis. It follows that $(X, L)$ can be degenerated to the
toric variety (together with a $\q$-divisor) associated to the polytope $\Delta_\S(X, L)$.
\end{COR}
This recovers toric degeneration results in \cite{AB}, \cite{Caldero} and \cite{Kaveh-Michigan}.

The interesting paper of \cite{AB} constructs certain toric
degenerations for projective spherical varieties of
any connected reductive algebraic group. This in particular includes all the flag varieties. Note that for
a given variety, in general there can exist many different toric
degenerations associated to different polytopes. As Alexeev and Brion point out, their
construction relies on Caldero's multiplicativity property of the
canonical basis and string parametrization (\cite[p.~4]{AB}) which is
purely representation theoretic. As far as the author knows, the
Alexeev-Brion result does not give a "geometric" interpretation of the
string parametrization (and hence a string polytope) for irreducible
representations of a reductive group. The main result in the present paper interprets the string polytope of an
irreducible representation as a Newton-Okounkov body associated to a natural
geometric valuation on the flag variety $G/B$.
 
One of the applications of the main result of the paper is that the toric degeneration
results of \cite{GL}, \cite{Caldero} and \cite{AB} (also several others including the author \cite{Kaveh-Michigan}) 
fit into and can be recovered form the more recent and (very)
general framework of toric degenerations associated to valuations and
Newton-Okounkov bodies. This general framework for toric degenerations has been treated systematically 
in \cite{Anderson} (see also the preprint \cite[Section 5.6]{KKh-affine}). The earlier paper of Teissier \cite{Teissier}
gives the most general setup for toric degenerations associated to valuations. All the above degenerations come from
degenerating an algebra to the associated graded of a filtration.

It is expected that the Gelfand-Cetlin and more generally the string polytopes carry a lot of information about the geometry of the flag variety (and more generally spherical varieties). 
In fact, there is a general philosophy that these polytopes play a role for the flag variety similar to the role of Newton polytopes for toric varieties. The results of this paper provide evidence in this direction. More evidence for this similarity is obtained in the recent work of 
V. Kiritchenko, E. Smironov and V. Timorin who made an interesting connection between the combinatorics of the faces of the Gelfand-Cetlin polytopes and Schubert calculus in type A (\cite{Valentina, KST}).

{ As pointed out by one of the referees it is interesting to investigate whether the main result of the paper works in the generality of Kac-Moody algebras and their flag varieties, because essentially all the statements used here hold in that setting. One missing piece is that in the general Kac-Moody setting there is no longest element in the Weyl group and hence no reduced decomposition for such a longest element. As the referee suggests probably the correct substitute for a reduced decomposition is a convex order (in the finite type, 
convex orders are in one-to-one correspondence with reduced decomposition for the longest element). In particular it would be interesting to extend the main theorem in the paper (Theorem \ref{th-main}) to affine type.}

To make the paper more accessible and easier to read we have tried to include much of the background material (Sections 
\ref{sec1}, \ref{sec2} and \ref{sec3}).

\vspace{.4cm}
\noindent{\bf Acknowledgment:} The author would like to thank Dave Anderson, Jim Carrell, Askold
Khovanskii, Valentina Kiritchenko, Tatiyana Firsova, Macej
Mizerski and Jochen Kuttler for helpful discussions. In particular, Jochen Kuttler provided the proof of Lemma \ref{lem-pure-tensor}. 
Moreover, I am much thankful to Joel Kamnitzer who explained the key properties of crystal bases to me. 
{ Finally, the author is grateful to the anonymous referees whose comments and remarks greatly improved the content and exposition of the paper.}\\

\noindent{\bf Notation:} Throughout the paper we will use the
following notation: 
\begin{itemize}
\item[-] $G$ is a connected reductive algebraic group over
$\c$, $B$ a Borel subgroup, $T$ { a maximal torus and $U$ the maximal unipotent subgroup of $B$.} 
\item[-] $B^{-}$ and $U^{-}$ are the opposite subgroups of $B$ and $U$ respectively.
\item[-] $\Phi = \Phi(X, T)$ denotes the root system with $\Phi^+ = \Phi^+(X, T)$ the subset of positive
roots for the choice of $B$.  
\item[-] $\alpha_1, \ldots, \alpha_r$ denote the simple roots where $r$ is the
semi-simple rank of $G$. 
\item[-] $W$ is the Weyl group of $(G,T)$. The simple reflection associated with
a simple root $\alpha$ is denoted by $s_\alpha$. 
\item[-] $w_0$ is the unique longest element in $W$. $N$ denotes the length of $w_0$ which is equal to the number of
positive roots as well as the dimension of the flag variety $G/B$.
\item[-] $E_\alpha$, $F_\alpha$ are the Chevalley generators for a root $\alpha$, which { are generators} for
the root subspaces $\Lie(G)_\alpha$ and $\Lie(G)_{-\alpha}$ respectively. 
\item[-] $U_\alpha = \{ \exp(tE_\alpha) \mid t \in \c \}$,  $U^-_{\alpha} = \{ \exp(tF_\alpha) \mid t \in \c \}$ 
denote the root subgroups corresponding to a root $\alpha$.
\item[-] {$\Lambda$ is the weight lattice of $G$. We denote the rank of $\Lambda$, equal to $\dim(T)$, by $n$. 
$\Lambda^+$ is the subset
of dominant weights and $\Lambda_\r = \Lambda \otimes_{\z} \r$. The cone generated by $\Lambda^+$ is the positive Weyl chamber denoted by $\Lambda^+_\r$.}
\item[-] $V_\lambda$ denotes the irreducible $G$-module corresponding to a dominant weight $\lambda$. Also
$v_\lambda$ denotes a highest weight vector in $V_\lambda$. 
\item[-] For a dominant weight $\lambda$, $-w_0\lambda$ is denoted by $\lambda^*$.  It is dominant and
$V_\lambda^* \cong V_{\lambda^*}$.
\item[-] $o = eB$ is the unique $B$-fixed point
in the flag variety $G/B$. 
\item[-] $C_w$, $X_w$ denote the
Schubert cell and the Schubert variety in $G/B$ corresponding to $w \in W$ respectively.
\end{itemize}

An $N$-tuple of simple roots $\S = (\alpha_{i_1}, \ldots, \alpha_{i_N})$ is called a
{\it reduced decomposition for the longest element $w_0$} if $w_0 = s_{\alpha_{i_1}} \cdots s_{\alpha_{i_N}}.$
\\
\section{Valuations and Newton-Okounkov bodies} \label{sec1}
{A general reference for material in this section is \cite{KKh-NO}.}
\subsection{Valuations} \label{sec-valuation}
Let $V$ be a vector space over $\c$ and let $\Gamma$ be a set with a total order $<$.

\begin{Def}[Pre-valuation]  \label{def-pre-valuation}
A function $v: V \setminus \{0\} \to \Gamma$ is a {\it pre-valuation with
values in $\Gamma$} if:
\begin{enumerate}
\item[(i)] $v(f+g) \geq \min\{v(f), v(g)\}$, for all nonzero $f, g \in V$.
\footnote{Some authors may use the axiom $v(f+g) \leq \max\{v(f), v(g)\}$ instead. It is
equivalent to ours by considering the reverse order on $\Gamma$.}
It follows that if $v(f) \neq v(g)$ then $v(f+g) = \min\{v(f), v(g) \}$.
\item[(ii)] $v(c f) = v(f)$, for all nonzero $f \in V$ and nonzero $c \in \c$.
\end{enumerate}
{ One often extends $v$ to the whole $V$ by defining $v(0) = \infty$. From the definition it follows that for any $a \in \Gamma$, the sets 
$\{f \mid v(f) > a\}$ and $\{f \mid v(f) \geq a\}$ are vector subspaces of $V$.}
For $a \in \Gamma$ consider the quotient vector space,
$$F_a = \{f \mid v(f) \geq a\} / \{f \mid v(f) > a\}.$$ We call this the {\it leaf at $a$}.
The pre-valuation $v$ is said to have {\it one-dimensional leaves}
if for any $a \in \Gamma$, the leaf $F_a$ is at most one-dimensional. Equivalently,
$v$ has one-dimensional leaves, if whenever $v(f) = v(g)$, for some $f,g \in V$,
then there is $c \neq 0$ such that $v(g - c f) > v(g)$. 
\end{Def}

Next, let $A$ be an algebra over $\c$ and
let $\Gamma$ be a commutative semigroup totally ordered with an ordering $<$ respecting
the semigroup operation (which we write additively). In this paper we will always assume that $\Gamma$
is a free abelian group of finite rank.

\begin{Def}[Valuation] \label{def-valuation}
A pre-valuation on $A$ with values in $\Gamma$ is called a {\it valuation}
if moreover it satisfies: $v(fg) = v(f) + v(g)$, for all nonzero $f, g \in A$.
\end{Def}

\begin{Def}
Given an algebra $A$ with a valuation $v$ with values in $\Gamma$, it is easy to see that
$$\{ v(f) \mid f \in A \setminus \{0\}\},$$ is a semigroup in $\Gamma$. One calls it
the {\it value semigroup of the pair $(A, v)$}.
\end{Def}

\begin{Ex} \label{ex-valuation-curve}
Let $X$ be an algebraic curve over $\c$ with { field of rational functions} $\c(X)$.
Let $p$ be a smooth point on $X$ and for any $0 \neq f \in \c(X)$ define
$v(f)$ to be the order of vanishing of $f$ at $p$ (zero or pole). Then $v$ is a valuation (with
one-dimensional leaves) on $\c(X)$ and with values in $\z$ (with the usual ordering of numbers).
\end{Ex}

The previous example generalizes to higher dimensional varieties:

\begin{Ex}[Highest term and lowest term valuations] \label{ex-Grobner-valuation}
{ Let $X$ be a $d$-dimensional variety over $\c$ with $\c(X)$ its field of rational functions. Given a smooth point $p$ on $X$ and a regular system of parameters $u_1, \ldots, u_d$
in a neighborhood of $p$, as discussed in the introduction, we can define a {\it lowest term valuation} $v$ on $\c(X)$. Fix an ordering on $\z^d$. 
Let $f(u_1, \ldots, u_d) = \sum_{j=(j_1, \ldots, j_d)} c_j u_1^{j_1} \cdots u_d^{j_d}$ be a polynomial in the $u_i$. Then $v(f) = (k_1, \ldots, k_d)$ where $k = (k_1, \ldots, k_d)$
is the smallest among $j=(j_1, \ldots, j_d)$ with $c_j \neq 0$. The {\it highest term valuation} on $\c(X)$ is 
defined similarly.}
\end{Ex}

\begin{Ex}[Parshin valuation] \label{ex-Parshin-valuation}
More generally, one can construct a valuation out of a sequence of subvarieties in $X$. Let
$$\{p\} = Y_d \subset \cdots \subset Y_0 = X $$ be a sequence of closed irreducible subvarieties in $X$
such that $\dim(Y_k) = d-k$ and $Y_k$ is non-singular along $Y_{k+1}$, { that is the local ring $\mathcal{O}_{Y_k, Y_{k+1}}$ is a regular local ring}, for all $k$.
Sometimes such a sequence of subvarieties is called a {\it Parshin point} on the variety $X$ (\cite{Parshin}).
A collection $u_1, \ldots, u_d$ of rational functions on $X$ { is said to be a} {\it system of parameters} about such a sequence
if for each $k$, ${u_k}_{|Y_k}$ is a well-defined not identically zero rational function on $Y_k$
and has a zero of first order on the
hypersurface  $Y_{k+1}$, {in other words, $u_k$ represents a generator of the maximal ideal of $\mathcal{O}_{Y_k, Y_{k+1}}$}. 
Given a sequence of subvarieties and a system of parameters
$u_1, \ldots, u_d$ as above, one defines a valuation $v$ on $\c(X)$ with one-dimensional leaves and with values in $\z^d$ (ordered lexicographically): 
Take $0 \neq f \in \c(X)$, then $v(f) = (k_1, \ldots, k_d)$ where the $k_i$ are defined as follows. $k_1$ is the order of vanishing of
$f$ on $Y_1$. Now $f_1 = (u_1^{-k_1}f)_{|Y_1}$ is a
well-defined, not identically zero rational function on $Y_1$. Then
$k_2$ is the order of vanishing of $f_1$ on $Y_2$ and so on
for $k_3$ etc. { (In fact, the assumption of non-singularity of the $Y_i$ along $Y_{i+1}$ is not crucial and one can avoid it by successive normalizations.)}
\end{Ex}

\begin{Rem} \label{rem-Grobner-val-Parshin-val}
Example \ref{ex-Grobner-valuation} is a special case of Example \ref{ex-Parshin-valuation} where for each $k>0$,
we take $Y_k$ to be (the irreducible component of) the zero locus of $\{u_1, \ldots, u_k\}$
containing $p$. Conversely, by taking an appropriate resolution of $X$ at $p$ and a suitable system of parameters
in the resolution, one can realize the valuation constructed out of a sequence of subvarieties as a valuation coming from a system 
of parameters at a smooth point. 
\end{Rem}

Next, we state { some easy general facts} about pre-valuations on vector spaces and valuations on algebras. For completeness we include the short proofs. Let $v$ be a pre-valuation on a vector space $V$ with values in a totally ordered set $\Gamma$.

\begin{Prop} \label{prop-valuation-subspace}
Suppose $f_1, \ldots, f_s \in V$ are such that $v(f_1), \ldots, v(f_s)$ are distinct. Then:
(1) the $f_i$ are linearly independent. (2) If $f = \sum_{i=1}^s c_i f_i$ with $0 \neq c_i \in \c$ then
$v(f) = \min \{ v(f_1), \ldots, v(f_s)\}$.
\end{Prop}
\begin{proof}
(1) Let $\sum_{i=1}^s f_i = 0$, be a nontrivial linear relation between the $f_i$.
Let $\alpha_i = v(f_i)$, $i=1, \ldots, s$. Without loss of generality assume $c_1 \neq 0$ and
$\alpha_1 < \cdots < \alpha_s$. We can rewrite the linear relation as $c_1 f_1 = -\sum_{i=2}^s c_if_i$.
But this is not possible as $v(c_1 f_1) = \alpha_1$ while $v(-\sum_{i=2}^s c_i f_i) \geq \alpha_2$. This
proves (1). Part (2) follows by induction on $s$ and using the property (i) in the definition of pre-valuation (Definition \ref{def-pre-valuation}).
\end{proof}

\begin{Prop} \label{prop-basis-distinct-valuation}
Let $V$ be finite dimensional, moreover assume that the valuation $v$ has one-dimensional leaves. 
(1) There exists a basis $\B$ for $V$
such that all the $v(b)$, $b \in \B$ are distinct. (2) $\dim(V) = \# v(V \setminus \{0\})$.
\end{Prop}
\begin{proof}
(1) { Suppose $\B = (b_1, \ldots, b_s)$ is an ordered basis for $V$ with 
$v(b_1) \lneqq \cdots \lneqq v(b_t) \leq \cdots \leq v(b_s)$. We will construct another ordered basis
$\B' = (b'_1, \ldots, b'_s)$ such that $v(b'_1) \lneqq \cdots \lneqq v(b'_{t+1}) \leq \cdots \leq v(b'_s)$}. If $v(b_t) \neq v(b_{t+1})$ then just take
$\B' = \B$. Otherwise, there is a scalar $c \neq 0$ such that $v(b_{t+1} - cb_t) \gneqq v(b_t)$.
Now replace $b_{t+1}$ with $b_{t+1} - cb_t$ and sort the new set in increasing order if necessary, to obtain the ordered basis $\B'$. Continuing this procedure we will arrive at a basis such that all the values of valuation on
the basis are distinct. Part (2) follows from (1) and Proposition \ref{prop-valuation-subspace}(1).
\end{proof}

Let $A$ be an algebra and let $\B$ be a vector space basis for $A$ over $\c$. Moreover assume that
the values of $v$ on $\B$ are all distinct. The following multiplicativity property is a { straightforward}
corollary of defining properties of a valuation.
\begin{Prop}[Multiplicativity property] \label{prop-multi-prop}
For any two $b'$, $b'' \in \B$, the element $b' b''$ can be written uniquely
as $$b' b'' = c b + \sum_i c_i b_i,$$ where
$b$ and the $b_i \in \B$, $c$ and the $c_i$ are nonzero scalars,
$v(b) = v(b') + v(b'')$ and $v(b_i) \gneqq v(b') + v(b'')$ for all $i$.
\end{Prop}
\begin{proof}
{ The proposition follows directly from} Proposition \ref{prop-valuation-subspace}(2).
\end{proof}

In Section \ref{sec-multi-prop} we will see that the multiplicativity property of the so-called
dual canonical basis for the algebra $\c[G/U]$ of a reductive group $G$, is a special case of the above.

\subsection{Newton-Okounkov bodies} \label{sec-Newton-Ok-bodies}
In this section we consider valuations on the rings of sections of line bundles. We recall from \cite{KKh-NO} 
the construction of a Newton-Okounkov body and the related results on Hilbert functions.

{ Let $X$ be a projective variety of dimension $d$ over $\c$ equipped
with a very ample line bundle $L$}.  
The line bundle $L$ gives rise to a \emph{Kodaira map $\Phi_L$}, 
which is an embedding of $X$ into the projective space $\p(E^*)$ where $E = H^0(X, L)$. 

The {\it ring of sections of $L$} is the graded algebra:
$$R(L) = \bigoplus_{k \geq 0} H^0(X, L^{\otimes k}),$$
{ One knows that $R(L)$ is a finitely generated algebra.} 

{ For the rest of this section $R = \bigoplus_{k \geq 0} R_k$ denotes a graded subalgebra of $R(L)$. 
Such an $R$ is usually referred to as a {\it graded linear series} on $X$.
The homogenous coordinate ring of $X$ embedded in $\p(E^*)$ is an example of $R$. Also 
the ring of sections of an arbitrary line bundle on $X$ can be realized as a graded subalgebra of $R(L)$ for some very ample 
line bundle $L$. For simplicity we will assume that $R_k \neq \{0\}$ for all sufficiently large $k$.}

Fix a valuation $v$ on $\c(X)$ with values in a totally ordered free abelian group $\Gamma$. 
Using $v$ one can associate a semigroup $S(R) \subset \n \times \Gamma$
to $R$. Fix a non-zero element $\tau \in H^0(X, L)$. We use $\tau$ to identify $H^0(X, L)$ with a finite-dimensional
subspace of $\c(X)$ by mapping $\sigma \in H^0(X, L)$ 
to the rational function $\sigma/\tau \in \c(X)$. Similarly,
we can associate the rational function $\sigma/\tau^k$ to an element 
$\sigma \in H^0(X, L^{\otimes k})$. Using these
identifications, we define 
\begin{equation}\label{equ-def-S}
S = S(R) = S(R,v,\tau) = \bigcup_{k > 0} \{ (k, v(\sigma / \tau^k)) \mid \sigma \in
R_k \setminus \{0\}\} \subset \n \times \Gamma.
\end{equation}
{ From the definition} of valuation it follows that $S(R)$ is an additive semigroup.

In order to keep track of the grading on the ring $R =
\bigoplus_{k \geq 0} R_k$, 
it is convenient to extend the valuation $v$ to a valuation $\tilde{v}: R \setminus \{0\} \to \z_{\geq 0} \times \Gamma$ as follows.
We define an ordering on $\z_{\geq 0} \times \Gamma$ by $(m, u) \leq (m', u')$ if
and only if 
\begin{equation}\label{equ-ordering-N-times-Z^n} 
\textup{ either ($m>m'$) \quad or \quad ($m=m'$ and $u \leq u'$) (note the switch!) }
\end{equation}
For $\sigma \in R$, we now define
\begin{equation}\label{equ-def-tilde-v} 
\tilde{v}(\sigma) = (m, v(\sigma_m/\tau^m))
\end{equation}
where $\sigma_m$ is the highest-degree homogeneous component of $\sigma$. 
One verifies the following:

\begin{Prop} \label{prop-tilde-v}
(1) The map $\tilde{v}$ is a valuation on $R$ with values in $\z_{\geq 0} \times \Gamma$ ordered with the above ordering. 
(2) If the valuation $v$ has one-dimensional leaves then $\tilde{v}$ also has one-dimensional leaves. 
(3) The value semigroup of $\tilde{v}$ is exactly the semigroup $S = S(R)$.
\end{Prop}

\begin{Rem} \label{rem-depend-valuation}
The semigroup $S = S(R)$ depends on the 
choice of valuation $v$ on $\c(X)$ and the section $\tau$. The dependence
on $\tau$ is minor; a different choice $\tau'$ would lead to a semigroup
which is shifted by the vector $k v(\tau/\tau')$ at the level $\{k\} \times \Gamma$. 
However, the dependence on the valuation $v$ is much more subtle. 
\end{Rem}

To $R$ one associates the following objects:
\begin{itemize}
\item[-] The cone $C = C(R) \subset \r \times \Gamma_\r$
generated by the semigroup $S = S(R)$,
i.e. the smallest closed convex cone { with apex at the origin} and containing $S$.
\item[-] The subgroup $\Lambda(R) \subset \Gamma$ which is the intersection of the group generated by $S$ with $\{0\} \times \Gamma \cong \Gamma$.
\end{itemize}

\begin{Def}[Newton-Okounkov body] \label{def-Newton-Okounkov-body}
Let $\Delta = \Delta(R) = \Delta(R, v, \tau)$ be the slice of the cone $C$ at $k=1$ projected to
$\Gamma_\r$, via the projection on the second factor $(k,x) \mapsto x$. In other words:
$$\Delta = \overline{\conv(\bigcup_{k>0} \{x/k \mid (k,x) \in S \})}.$$
We call $\Delta$ the {\it Newton-Okounkov body of $R$}. 
(It can be shown that $\Delta$ is bounded and hence a convex body.)
\end{Def}

\begin{Rem}
Even when $R$ is finitely generated, the semigroup $S$ may not be finitely generated and the
cone $C$ may not be polyhedral. Thus, in general, the convex body $\Delta$ is not necessarily a polytope.
Although in most of the examples appearing in this paper, the Newton-Okounkov bodies turn out to be polytopes.
\end{Rem}

The Newton-Okounkov body $\Delta(R)$ encodes information about the 
asymptotic behavior of the Hilbert function of $R$.
Let $H_R(k) := \dim_\c(R_k)$ be the Hilbert function of the graded algebra $R$. { 
Below we assume $v$ to be a valuation on $\c(X)$ with values in $\z^d$ and with one-dimensional leaves.}

\begin{Th}[\cite{Ok-LC, LM, KKh-affine, KKh-NO}] \label{th-asymp-Hilbert-Okounkov-body}
{ Let $q \leq d$ be the (real) dimension of the Newton-Okounkov body $\Delta(R)$. Then $H_R$ grows 
of order $q$ and the growth coefficient $$a_q = \lim_{k \to \infty} \frac{H_R(k)}{k^q},$$ 
of the Hilbert function of $R$ is equal to $\Vol_q(\Delta(R))$. Here $\Vol_q$ is the
$q$-dimensional Lebesgue measure in the affine space spanned by $\Delta(R)$ and normalized with respect to the lattice $\Lambda(R)$. In particular the degree
of the projective embedding of $X$ in $\p(E^*)$ is equal to $d! \Vol_d(\Delta(R))$.}
\end{Th}

\begin{Rem}
{ In the above theorem the assumption that the valuation has one-dimensional leaves is crucial.}
\end{Rem}

\section{Schubert and Bott-Samelson varieties} \label{sec2}
Let $G$ be a connected reductive algebraic group over $\c$. We will follow the notation introduced
after the introduction. { Some references for material in this section are \cite{Humphreys}, \cite{Jantzen}, \cite{Magyar} and \cite[Sections 2.1-2.2]{Brion-Kumar}.}
\subsection{Sequence of Schubert varieties associated to a reduced decomposition}
\label{sec-Schubert}
Let $\S = (\alpha_{i_1}, \ldots, \alpha_{i_N})$ be a reduced decomposition for the longest
element $w_0 \in W$, that is $$w_0 = s_{\alpha_{i_1}} \cdots s_{\alpha_{i_N}},$$
where $N = \ell(w_0)$ and $s_{\alpha_j}$ is the simple reflection corresponding to a simple root $\alpha_{j}$.
For the rest of the paper we fix a reduced decomposition $\S$.

For $k=0, \ldots, N,$ put $$w_k = s_{\alpha_{i_{k+1}}} \cdots s_{\alpha_{i_{N}}},$$ ($w_N$ is the identity $e$). Since $\S$ is a reduced decomposition we have
$\ell(w_k) = N-k$, and
$$e = w_N < w_{N-1} < \cdots  < w_0,$$
in the Bruhat order. Let $X_k = X_{w_k}$ be the Schubert variety corresponding to $w_k \in W$.
We have the sequence:
$$\{o\} = X_N \subset X_{N-1} \subset \cdots \subset X_0 = G/B,$$
with $\dim(X_k) = N-k$.
One verifies that, for $k=0, \ldots, N-1$, the 
Schubert variety $X_k$ is invariant { under a representative of} $s_{\alpha_{i_{k+1}}}$, and 
hence under the opposite root subgroup $U^-_{\alpha_{i_{k+1}}}$. We denote by 
$\V_{\alpha_{i_{k+1}}}$ the generating vector field of the Lie algebra element $F_{\alpha_{i_{k+1}}}$ on $X_k$. 

To the reduced decomposition $\S$ one can also associate a sequence of translated Schubert varieties 
$Y_{\bullet}$ as follows: For each $k$, let $Y_k = w_0w_k^{-1}X_k$ be the Schubert variety of $w_k$ translated by $w_0w_k^{-1}$.
Since $X_k$ is invariant under a representative of $s_{\alpha_{i_{k+1}}}$, we see that $Y_{k+1} \subset Y_k$, for 
$k=0, \ldots N-1$. Thus we have a sequence:
\begin{equation} \label{eqn-flag-of-trans-Schubert-var}
\{w_0o\} = Y_N \subset \cdots \subset Y_0 = G/B.
\end{equation}
As the Schubert varieties are normal and irreducible, (\ref{eqn-flag-of-trans-Schubert-var}) 
gives a sequence of normal subvarieties in $G/B$. { Also since $w_0 w_k^{-1}B w_k o$ is an open set in $Y_k$ containing $w_0o$ and isomorphic to an affine space, 
we see that, unlike the Schubert varieties $X_k$ which are most likely singular at 
the point $o$, the translated Schubert varieties $Y_k$ are smooth at the point $w_0 o$}.

In the next section we construct { a system of parameters corresponding to the sequence of Schubert subvarieties $X_\bullet$ and consider the corresponding highest and lowest term valuations}. For this purpose it is natural to consider the Bott-Samelson 
resolution of singularities $X_\S$ associated to the reduced decomposition $\S$.

\subsection{Bott-Samelson variety associated to a reduced decomposition}
\label{sec-Bott-Samelson}
For each reflection $s_\alpha \in W$ fix a representative $\dot{s}_\alpha \in N(T)$.
For a simple root $\alpha$ let $P_\alpha$ denote the {\it minimal standard parabolic
subgroup of $\alpha$}, that is, the subgroup generated by $B$ and $\dot{s}_\alpha$
(clearly it is independent of the choice of the representative $\dot{s}_\alpha$). One
verifies that $P_\alpha = B \cup (\dot{s}_\alpha B \dot{s}_\alpha)$. 
Also one shows that $P_\alpha / B$ is isomorphic to $\c P^1$ and the map $x \mapsto xB$ gives an embedding of $U^-_{\alpha}$ into $P_\alpha / B \cong \c P^1$ as an open neighborhood of
$eB$ isomorphic to $\c$.

Let $\w = (\alpha_{i_1}, \ldots, \alpha_{i_d})$ be a $d$-tuple of simple roots. The
Bott-Samelson variety $\X_\w$ associated to $\w$ is defined as
$$\X_\w = (P_{\alpha_{i_1}} \times \cdots \times P_{\alpha_{i_d}}) / B^d,$$ where
$B^d$ acts on $P_{\alpha_{i_1}} \times \cdots \times P_{\alpha_{i_d}}$ { from the right by}:
$$(p_1, \ldots, p_d) \cdot (b_1, \ldots, b_d) = (p_1b_1, b_1^{-1}p_2b_2, \ldots, b_{d-1}^{-1}p_db_d).$$
This is a smooth projective variety of dimension $d$, and multiplication defines a morphism
$\pi_\w: \tilde{X}_\w \to G/B$. Let $w = s_{\alpha_{i_1}} \cdots s_{\alpha_{i_d}}$. Suppose $\w$ is a reduced word, i.e.
$\ell(w) = d$. Then it is well-known that:
\begin{Th} \label{th-Bott-Samelson}
The map $\pi_\w: \tilde{X}_\w \to G/B$ is birational onto its image, which is the Schubert variety
$X_w$. Thus the Bott-Samelson varieties resolve singularities of Schubert varieties. Moreover,
$\pi_\w$ is $B$-equivariant, and is an isomorphism over the open Schubert cell $C_w$.
\end{Th}

Now fix a reduced decomposition $\S = (\alpha_{i_1}, \ldots, \alpha_{i_N})$ 
for the longest element $w_0$.
For $0 \leq k \leq N$, let $\underline{w_k} = (\alpha_{i_{k+1}}, \ldots, \alpha_{i_N})$ and 
$w_k = s_{\alpha_{i_{k+1}}} \cdots s_{\alpha_{i_N}}$ (if $k=N$, 
$\w_N$ is empty and $w_N = e$). Define $\X_{k}$ to be the set of 
$(p_1, \ldots, p_N) \mod B^d$ in $\X_\S$ such that $p_j = e$ if $1 \leq j \leq k$.
One sees that $\X_{k}$ is a subvariety of $\X_\S$ isomorphic to the Bott-Samelson variety 
$\X_{\underline{w_k}}$. The parabolic subgroup $P_{\alpha_{k+1}}$ acts on $\tilde{X}_k \cong \tilde{X}_{\underline{w_k}}$ by multiplication
from the left. Consider the map $\tilde{\Phi}_\S: \c^N \to U^-_{\alpha_{i_1}} \times \cdots \times U^-_{\alpha_{i_N}}
\to X_\S$ given by:
$$\tilde{\Phi}_\S: (t_1, \ldots, t_N) \mapsto (\exp(t_1 F_{\alpha_{i_1}}), \ldots, \exp(t_N F_{\alpha_{i_N}})) 
\mod B^N.$$ 

Note that for any $\alpha$ the minimal parabolic $P_{\alpha}$ contains $s_\alpha$ and $U_\alpha$, and hence contains the 
opposite root subgroup $U^-_\alpha$. Thus for any 
$k=0, \ldots, N-1$, we have an action of $U^-_{\alpha_{i_{k+1}}}$
on the Bott-Samelson variety $\X_k$. This induces a generating vector field $\tilde{\V}_{\alpha_{i_{k+1}}}$ on $\X_k$. 
Under the product map $\pi_\S$, $\tilde{\V}_{\alpha_{i_{k+1}}}$ goes to $\V_{\alpha_{i_{k+1}}}$.
{ One has the following:}
\begin{Prop}  \label{prop-Bott-Samelson-coord} ~
\begin{enumerate}
\item $\tilde{\Phi}_\S$ is an { open embedding and its image is a 
neighborhood $\tilde{\U}$ of the identity in $\X_\S$} (which necessarily intersects all the 
subvarieties $\X_{k}$).
\item Let $t_1, \ldots, t_N$ be the coordinates on the open subset $\tilde{\U} \subset \X_\S$ given by the embedding 
$\tilde{\Phi}_\S$. Then in $\tilde{\U}$, each subvariety $\X_{k}$ is given by 
$t_1 = \cdots = t_k = 0$.
\item In these coordinates the vector field $\tilde{\V}_{\alpha_{i_{k+1}}}$ on $\X_k$ is given by $\partial / \partial t_{k+1}$.
\end{enumerate}
\end{Prop}

\begin{proof}
{ Part (1) can be found in \cite[Section 1.4]{Magyar}. Part (2) is immediate from the definition of $\tilde{X}_\S$. 
To prove (3) it suffices to note that the action of $U_{\alpha_{i_{k+1}}}$ on $\tilde{X}_k \cap \tilde{\U}$ is 
given, in the coordinates $(t_{k+1}, \ldots, t_N)$, by: $$\exp(s F_{\alpha_{i_{k+1}}}): (t_{k+1}, \ldots, t_N) \mapsto (s + t_{k+1}, \ldots, t_N).$$
}
\end{proof}

Take the lexicographic order $t_1 > \ldots > t_N$ on $\z^N$. As in Example \ref{ex-Grobner-valuation} the coordinate system $t_1, \ldots, t_N$ 
gives rise to two different valuations on the field of rational functions $\c(\tilde{X}_\S)$: 
(1) the highest term valuation which we denote by $v_\S$, and (2) the lowest term valuation.
The main result of the paper (Theorem \ref{th-main}) states that { the highest term valuation $v_\S$ restricts to the string parametrization $\iota_\S$.}

\begin{Rem} \label{rem-Anderson}
{ From Proposition \ref{prop-Bott-Samelson-coord} we see that, as rational functions on 
$\c(G/B)$, the $t_k$ define a system of parameters (in the sense of Example \ref{ex-Parshin-valuation}) 
for the sequence of Schubert varieties $X_k$ (defined in Section \ref{sec-Schubert}).} 
The valuation on $\c(G/B)$ corresponding to this sequence and the 
system of parameters $t_k$ coincides with the lowest term valuation on $\c(\X_\S)$, via the birational isomorphism $\pi_\S$.
Similarly one defines $\Y_{k}$ to be the set of $(p_1, \ldots, p_N) \mod B^N$ in $\X_\S$ such that $p_j = \dot{s}_{\alpha_{i_j}}$ 
whenver $1 \leq j \leq k$. One verifies that this is independent of the choice of the representative $\dot{s}_{\alpha_{i_j}} \in N(T)$ and 
$\Y_{k}$ is a subvariety of $\X_\S$ isomorphic to $\X_{\underline{w_k}}$. 
Hence we have a sequence of translated Bott-Samelson varieties $$\tilde{Y}_{N} \subset \cdots 
\subset \tilde{Y}_{1} \subset \tilde{Y}_{0}.$$
For each $0 \leq k \leq N-1$ consider the embedding 
$\c^{N-k} \to U^-_{\alpha_{i_{k+1}}} \times \cdots \times U^-_{\alpha_{i_N}}  \to \Y_{k},$ given by 
\begin{equation} \label{equ-coor-Y}
(t_{k+1}, \ldots, t_N) \mapsto (\dot{s}_{\alpha_1}, \ldots, \dot{s}_{\alpha_k}, 
\exp(t_{k+1}F_{\alpha_{i_{k+1}}}), \ldots,  \exp(t_NF_{\alpha_{i_N}})) \mod B^N.
\end{equation}
Similar to Proposition \ref{prop-Bott-Samelson-coord} this embedding gives a system of 
coordinates on an open subset of $\Y_{k}$ which by abuse of notation we denote again by
$t_{k+1}, \ldots, t_N$. It is easy to verify that $t_{k+1}$, regarded as a rational function on $\Y_{k}$, has a pole of order $1$ on 
the hypersurface $\Y_{k+1}$. Hence we see that, as rational functions on $\c(G/B)$ via the birational map $\pi_{\S}$, 
the $u_k = 1/t_k$ define a system of parameters for the sequence of subvarieties $Y_\bullet$ (see \eqref{eqn-flag-of-trans-Schubert-var}).
Based on this fact, in an earlier version of this manuscript the author had claimed that the (highest term) valuation $v_\S$ on $\c(\X_\S)$ 
coincides, via the birational isomorphism $\pi_\S$, 
with the (lowest term) valuation $v_{Y_\bullet}$ on $\c(G/B)$ corresponding to the sequence $Y_\bullet$ of translated Schubert varieties and their local system of parameters $u_k = 1/t_k$. This statement seems not to be true. As pointed out by Dave Anderson, for $k \neq \ell$, the function $u_k$ may still vanish on 
$Y_\ell$ as can be computed in the example of Section \ref{sec-example} namely $G=\SL(3, \c)$ and 
$w_0 = s_\alpha s_\beta s_\alpha$. 
The Newton-Okounkov body associated to this data and $v_{Y_\bullet}$ is computed in \cite{Anderson}. It is a polytope combinatorially equivalent 
to the string polytopes for $\SL(3, \c)$ and $w_0 = s_\alpha s_\beta s_\alpha$. { Surprisingly, it can be verified by direct calculation 
that} there is no upper triangular linear change of coordinates in $\z^3$ that maps Anderson's polytope to the string polytope, which should be 
the case if the two valuations coincided. Note that the 
valuation $v_{Y_\bullet}$ corresponding to the sequence of subvarieties $Y_\bullet$ is defined up to an upper triangular change of coordinates in $\z^N$ corresponding to different choices of parameters for the $Y_k$. 
{ As another example one can take $G = \Sp(4, \c)$. The Newton-Okounkov body $\Delta(R(L_\lambda), v_{Y_\bullet})$ for $G = \Sp(4, \c)$ and the valuation
$v_{Y_\bullet}$ corresponding to the reduced decomposition $\S = (s_2, s_1, s_2, s_1)$ is calculated in \cite[Section 3.4]{Valentina-DDO}. Here
$\alpha_1$ denotes the shorter root and $\alpha_2$ is the longer root.
It is shown that when $\lambda$ is a strictly dominant weight, the Newton-Okounkov body is a polytope with $11$ vertices. 
On the other hand one can show that the string polytope $\Delta_\S(\lambda)$ has $12$ vertices and hence is not even combinatorially equivalent to $\Delta(R(L_\lambda), v_{Y_\bullet})$.}
\end{Rem}
\section{Crystal bases and their string parametrization} \label{sec3}
\subsection{Perfect bases and crystal bases} \label{sec-crystal}
{ In this section we recall some background material about the crystal bases and crystal graphs of representations. Some references are \cite{Kashiwara-crystal-notes},
\cite{HK}, \cite{BK} and \cite{KKKS}.}

Let $V$ be a finite dimensional $G$-module.
Let $\alpha$ be a simple root with the corresponding { root space generators $E_\alpha$
and $F_\alpha$.} Define the functions $\epsilon_\alpha, \varphi_\alpha: V\setminus \{0\}
\to \z,$ by: 
\begin{eqnarray} \label{equ-def-epsilon-phi}
\epsilon_\alpha(v) &=& \max \{a \mid E_\alpha^a \cdot v \neq 0\}, \cr
\varphi_\alpha(v) &=& \max\{a \mid F_\alpha^a \cdot v \neq 0\}.
\end{eqnarray}
If $E_\alpha \cdot v = 0$ (respectively $F_\alpha \cdot v = 0$)
we put $\epsilon_\alpha(v) = 0$ (respectively $\varphi_\alpha(v)=0$).
One knows that there is a vector space basis $\B_V$ for $V$ consisting of weight vectors and with the following properties (\cite[Lemma 5.50]{BK}):
For each $b \in \B_V$ let $E_\alpha \cdot b = \sum_i c_{i}b_i$ with $b_i \in \B_V,c_{i} \neq 0$. Then
\begin{enumerate}
\item[(i)] For every $b_i$, $\epsilon_\alpha(b_i) \leq \epsilon_\alpha(b) - 1$.
\item[(ii)] Provided that $E_\alpha \cdot b \neq 0$, there exists a unique $k$ with $\epsilon_\alpha(b_k) =
\epsilon_\alpha(b) - 1$. For all other $i \neq k, ~~ \epsilon_\alpha(b_i) <
\epsilon_\alpha(b) - 1.$
\end{enumerate}
and similarly, for each $b \in \B_V$ let $F_\alpha \cdot b = \sum_j e_{j}b_j, \, e_{j} \neq 0$. Then
\begin{enumerate}
\item[(iii)] For every $b_j$, $\varphi_\alpha(b_j) \leq \varphi_\alpha(b) - 1$.
\item[(iv)] Provided that $F_\alpha \cdot b \neq 0$, there exists a unique $\ell$ with $\varphi_\alpha(b_\ell) =
\varphi_\alpha(b) - 1$. For all other $j \neq \ell, ~~ \varphi_\alpha(b_j) <
\varphi_\alpha(b) - 1.$

Finally, for $b \in \B_V$, define $\tilde{E}_\alpha(b) = b_k$ and $\tilde{F}_\alpha(b) = b_\ell$.
If $E_\alpha(b) = 0$ (respectively $F_\alpha(b) = 0$) let $\tilde{E}_\alpha(b) = 0$
(respectively $\tilde{F}_\alpha(b) = 0$).
\item[(v)] For $b, b' \in \B_V$, $\tilde{E}_\alpha(b) = b'$ if and only if $\tilde{F}_\alpha(b') = b$.
\end{enumerate}

{
\begin{Def} \label{def-perfect-basis}
A basis $\B_V$ which satisfies the above { is called a {\it perfect basis for the representation $V$}}.
The operators $\tilde{E}_\alpha, \tilde{F}_\alpha: \B_V \to \B_V \cup \{0\},$
are called the {\it Kashiwara operators corresponding to the simple root $\alpha$} (\cite[Definition 5.49]{BK}).
\end{Def}

Consider the directed labeled graph whose vertices are the elements of $\B_V \cup \{0\}$
and its directed edges are labeled by the simple roots in the following way: for $b, b' \in \B_V$
we have $b \xrightarrow{\alpha} b'$ if $\tilde{E}_\alpha(b) = b'$ (equivalently $\tilde{F}_\alpha(b') = b$).
Also for $b \in \B_V$ we write
$b \xrightarrow{\alpha} 0$ if $\tilde{E}_\alpha (b) = 0$, and $0 \xrightarrow{\alpha} b$ if $\tilde{F}_\alpha(b) = 0$.
It is known that for different bases satisfying the above conditions, the graphs produced are isomorphic (\cite[Theorem 5.55]{BK}).
This graph is called the {\it crystal graph of the representation $V$}.}
 
For each $\lambda \in \Lambda^+$, consider 
the basis $\B_\lambda$ of $V_\lambda$ consisting of the nonzero $bv_\lambda$, where $b$ lies in the specialization
at $q = 1$ of the Lusztig canonical basis (Kashiwara lower global basis) for the quantum enveloping algebra. One knows that the dual basis $\B_\lambda^*$ is a perfect basis
for the dual representation $V^*_\lambda$ (\cite[Lemma 5.50]{BK} and \cite[Proposition 4.3 and Proposition 7.3]{KKKS}). Throughout the paper we call this basis the {\it dual crystal basis}.
\footnote{As pointed out before, in the literature what we call a crystal basis is usually called a canonical basis (see the footnotes in page 2).} 
We will use some extra properties of this basis (than just being a perfect basis) namely compatibility with the Demazure modules discussed in Section \ref{sec-Demazure}.

\subsection{String parametrization} \label{sec-string-para}
In this section we define the string parametrization for the elements of a crystal basis of a representation.
In \cite{Littelmann} and \cite{BZ-CB} the authors construct a remarkable parametrization, called
the {\it string parametrization}, for the
elements of a crystal basis by the integral points in certain polytopes. The construction depends on the
choice of a reduced decomposition $\S$ for the longest element $w_0 \in W$.
{In this paper we realize the { dual irreducible representations $V^*_\lambda$ as representations in spaces of functions.} Hence
we will discuss the string parametrization for the dual crystal bases.

Fix a reduced decomposition $\S = (\alpha_{i_1}, \ldots, \alpha_{i_N})$,
$w_0 = s_{\alpha_{i_1}} \cdots s_{\alpha_{i_N}}$.

\begin{Def}[String parametrization] \label{def-string-para-dual-basis}
{ Define the map $\iota_\S: \B^*_\lambda \to \z_{\geq 0}^N$ by
$\iota_\S(b^*) = (a_1, \ldots, a_N)$, where the $a_i$ are defined inductively by:}
\begin{eqnarray*}
a_1 &=& \max \{a \mid \tilde{F}_{\alpha_{i_1}}^a (b^*) \neq 0\}, \cr
a_2 &=& \max \{a \mid \tilde{F}_{\alpha_{i_2}}^a \tilde{F}_{\alpha_{i_1}}^{a_1} (b^*) \neq 0\}, \cr
a_3 &=& \max \{a \mid \tilde{F}_{\alpha_{i_3}}^a \tilde{F}_{\alpha_{i_2}}^{a_2}
\tilde{F}_{\alpha_{i_1}}^{a_1} (b^*) \neq 0\}, \quad \textup{etc.} \cr
\end{eqnarray*}
We call the map $\iota_\S$ the {\it string parametrization of $\B_\lambda^*$ corresponding to the reduced decomposition $\S$}.
\end{Def}

\begin{Rem}
{ In the above definition we can choose any perfect basis for $V_\lambda^*$ in place of the dual crystal basis $\B_\lambda^*$.
As the crystal graph is independent of the { choice of a perfect basis}, the image of the string parametrization is also independent of this choice.}
\end{Rem}

The following remarkable result due to Littelmann describes the image of the
string parametrization for finite dimensional irreducible $G$-modules.
{ For $\lambda \in \Lambda^+$ let}
$S_\lambda$ denote the image of the dual basis $\B^*_\lambda$ under the string parametrization $\iota_\S$.
\begin{Th} \cite[Proposition 1.5]{Littelmann} \label{th-Littelmann} 
\begin{enumerate}
\item For any dominant weight $\lambda$, $\dim(V^*_\lambda) = \#S_\lambda$, i.e. the
string parametrization is one-to-one.
\item Consider
$$\mathcal{S}_\S = \bigcup_{\lambda \in \Lambda^+} \{ (\lambda, a) \mid a \in S_\lambda\} \subset \Lambda^+ \times \z_{\geq 0} ^N.$$
Then $\mathcal{S}_\S$ is the intersection of a rational
{ closed} convex polyhedral cone $\mathcal{C}_\S$ in $\Lambda_\r \times \r^N$ with the
lattice $\Lambda \times \z^N$. (In particular, $\mathcal{C}_\S$ intersects { the subspace} $\{0\} \times \r^N$
only at the origin.)

\end{enumerate}
\end{Th}

\begin{Def}[String polytope]
For any $\lambda$ in the positive Weyl chamber $\Lambda^+_\r$, the {\it string polytope} $\Delta_\S(\lambda) \subset \r^N$
is the slice of the cone $\mathcal{C}_\S$ at $\lambda$, that is,
$$\Delta_\S(\lambda) = \{ a \mid (\lambda, a) \in \mathcal{C}_\S \}.$$
As $\mathcal{C}_\S$ is a rational convex polyhedral cone and intersects $\{0\} \times \r^N$ only at the origin, $\Delta_\S(\lambda)$ is a
rational convex polytope (i.e. with rational vertices).
\end{Def}

\begin{Rem} \label{rem-string-polytop}
{ (1) By Theorem \ref{th-Littelmann}, when $\lambda \in \Lambda^+$}, 
the number of integral points in the string polytope $\Delta_\S(\lambda)$ is equal to $\dim(V_\lambda)$.
(2) For any $k>0$, $\Delta_\S(k\lambda) = k\Delta_\S(\lambda)$.
\end{Rem}

\begin{Rem} \label{rem-string-G-C}
In \cite{Littelmann} it is shown that when $G = \SL(n, \c)$ and for a natural choice of
a reduced decomposition $\S$, after a fixed linear change of parameters,
the string polytope $\Delta_\S(\lambda)$ coincides with the Gelfand-Cetlin polytope of $\lambda$.
Similar statements hold for the Gelfand-Cetlin polytopes of the classical groups $\Sp(2n, \c)$ and $\SO(n, \c)$
(as introduced in \cite{BZ-PS}).
\end{Rem}

The following states that in defining the string parameters we can use
$F_\alpha$ instead of $\tilde{F}_\alpha$.
It is a straightforward corollary of the defining properties of a { perfect basis.}
\begin{Prop} \label{prop-string-para-F-alpha}
{ Let $b^* \in \B^*_\lambda$.}
For any simple root $\alpha$ we have $\tilde{F}_\alpha (b^*) = 0$
if and only if $F_\alpha \cdot b^* = 0$, { and more generally:}
$$\max \{ a \mid F_\alpha^a \cdot b^* \neq 0 \} = \max\{ a \mid \tilde{F}_\alpha^a (b^*) \neq 0\}.$$ 
Also if $k = \varphi_\alpha(b^*) = \max\{a \mid \tilde{F}_\alpha^a (b^*) \neq 0\}$ then $F_\alpha^k \cdot b^* = c \tilde{F}_\alpha^k (b^*)$ for some nonzero scalar $c$.
It follows that if $\iota_\S(b^*) = (a_1, \ldots, a_N)$ are the string parameters of $b^*$, then:
\begin{eqnarray*}
a_1 &=& \max \{a \mid {F}_{\alpha_{i_1}}^a \cdot b^* \neq 0\}, \cr
a_2 &=& \max \{a \mid {F}_{\alpha_{i_2}}^a {F}_{\alpha_{i_1}}^{a_1} \cdot b^* \neq 0\}, \cr
a_3 &=& \max \{a \mid {F}_{\alpha_{i_3}}^a {F}_{\alpha_{i_2}}^{a_2}
{F}_{\alpha_{i_1}}^{a_1} \cdot b^* \neq 0\}, \quad \textup{etc.} \cr
\end{eqnarray*}
\end{Prop}
\begin{proof}
{ The first assertion, i.e. $\tilde{F}_\alpha (b^*) = 0$ if and only if $F_\alpha \cdot b^* = 0$, is immediate from the definition of a perfect basis.
We prove the next assertion by induction on $k$. By the definition of the Kashiwara operator $\tilde{F}_\alpha$ we have: 
$$F_\alpha \cdot b^* = c_0 \tilde{F}_\alpha (b^*) + \sum_{i} c_i b^*_i,$$ where $c_0$ and the $c_i$ are nonzero scalars and 
$\varphi_\alpha(b^*_i) < k-1$ for all $i$. Thus 
$F_\alpha^{k-1}(F_\alpha \cdot b^*) = F_\alpha^{k-1}(c_0 \tilde{F}_\alpha (b^*))$.
Now by induction hypothesis we get:
$$F_\alpha^{k-1}(c_0\tilde{F}_\alpha (b^*)) = c \tilde{F}_\alpha^{k-1}(\tilde{F}_\alpha (b^*)) = c \tilde{F}_\alpha^k (b^*),$$ for some nonzero scalar $c$. This proves the claim.
}
\end{proof}

One can then extend the definition of the string parametrization to all the vectors in the $G$-module $V_\lambda^*$.
\begin{Def}[String parameters for an arbitrary vector] \label{def-string-para-all-vectors}
{ Let $\sigma \in V_\lambda^* \setminus \{0\}$}. Define the string parameters
$\iota_\S(\sigma) = (a_1, \ldots, a_N)$ as follows:
\begin{eqnarray*}
a_1 &=& \max \{a \mid {F}_{\alpha_{i_1}}^a \cdot \sigma \neq 0\}, \cr
a_2 &=& \max \{a \mid {F}_{\alpha_{i_2}}^a {F}_{\alpha_{i_1}}^{a_1} \cdot \sigma \neq 0\}, \cr
a_3 &=& \max \{a \mid {F}_{\alpha_{i_3}}^a {F}_{\alpha_{i_2}}^{a_2}
{F}_{\alpha_{i_1}}^{a_1} \cdot \sigma \neq 0\}, \quad \textup{etc.} \cr
\end{eqnarray*}
\end{Def}

Finally, consider the {\it generalized Pl\"{u}cker map} $\Phi_\lambda: G/P \to \p(V_\lambda)$, given by 
$gB \mapsto [g \cdot v_\lambda]$, where $v_\lambda$ is a highest weight vector in $V_\lambda$, $[v]$
denotes the point in the projective space represented by a vector $v$, and $P$ is the parabolic subgroup which is 
the $G$-stabilizer of $[v_\lambda]$ in $\p(V_\lambda)$.
If $\lambda$ is a regular dominant weight (i.e. lies in the interior of the positive
Weyl chamber) then $P = B$.
From Remark \ref{rem-string-polytop} we obtain:
\begin{Cor} \label{cor-deg-G/P}
The degree of the image of $G/P$ in the projective space $\p(V_\lambda)$ is equal to $m! \Vol_m(\Delta_\S(\lambda))$ 
where $m = \dim(G/P)$ and $\Vol_m$ is the Lebesgue measure in the real span of the $m$-dimensional polytope 
$\Delta_\S(\lambda) \subset \r^N$ normalized with respect to the lattice $\z^N$. In particular, when $\lambda$ is a regular dominant 
weight the degree of $G/B$ is equal to $N! \Vol_N(\Delta_\S(\lambda))$ where 
$\Vol_N$ is the standard $N$-dimensional Lebesgue measure in $\r^N$.
\end{Cor}
\begin{proof}
The homogeneous coordinate ring of the image of $G/P$ in $\p(V_\lambda)$ is isomorphic to
the graded ring $\bigoplus_{k \geq 0} V^*_{k\lambda}$. By Remark \ref{rem-string-polytop},
the dimension of the $k$-th graded piece is given by $\#(k\Delta_\S(\lambda) \cap \z^n)$. The corollary
follows from the Hilbert theorem on the degree of a projective subvariety of the projective space.
\end{proof}

The main result (Theorem \ref{th-main}) shows that Corollary \ref{cor-deg-G/P} is in fact 
a special case of a much more general theorem (Theorem \ref{th-asymp-Hilbert-Okounkov-body}) 
about Newton-Okounkov bodies.

\subsection{Demazure modules and string parametrization} \label{sec-Demazure}
{ In this section we cover some background material about Demazure 
modules and compatibility of crystal bases with Demazure modules.}
{ For general references on material in this section see \cite{Kumar} and \cite{Kashiwara-Demazure}.}

The purpose of considering Demazure modules in this section is to give a slightly different definition of the string parametrization
(Theorem \ref{th-alt-def-string-para}). This will be an important step in the proof of our main theorem (Theorem \ref{th-main}). 

Let $\lambda$ be a dominant weight.
For any $w \in W$ one knows that the weight space of the weight $w\lambda$ in $V_\lambda$ is
$1$-dimensional. An eigenvector $v_{w\lambda}$ of weight $w\lambda$ is
called an {\it extremal weight vector}. The $B$-module generated by $v_{w\lambda}$
is called the {\it Demazure module} corresponding to $w$ and $\lambda$ and denoted by $V_\lambda(w)$.
Note that the Demazure module $V_{\lambda}(w_0)$ is just the whole space $V_\lambda$.
Demazure modules play an important role in representation theory of $V_\lambda$ as well as in 
Schubert calculus.

For $w \in W$, the inclusion $V_\lambda(w) \subset V_\lambda$ induces a projection
$\pi_{w}: V_\lambda^* \to V_\lambda(w)^*$. It is known  
that for any $w \in W$, the restriction map $H^0(G/B, L_\lambda) \to H^0(X_w, {L_\lambda}_{|X_w})$ is
surjective and one can identify $H^0(X_w, {L_\lambda}_{|X_w})$ with the dual
Demazure module $V_\lambda(w)^*$. Under this identification the projection
$\pi_{w}$ corresponds to the restriction map $H^0(G/B, L_\lambda) \to H^0(X_{w}, {L_\lambda}_{|X_w})$.
It is also well-known (\cite{Kashiwara-Demazure}) that for any $w \in W$ there is a subset
$\B_\lambda(w)$ of the crystal basis $\B_\lambda$ which is a basis for $V_\lambda(w)$.
Moreover, $\B_\lambda(w) \cup \{0\}$ is invariant under $\tilde{E}_\alpha$, for any simple root $\alpha$.
As before let $\B_\lambda^*$ be the dual crystal basis for $V_\lambda^*$ and for each $b \in \B_\lambda$ let $b^* \in \B_\lambda^*$
be its corresponding dual basis element. { One then knows the following 
(see \cite[Section 1.8]{Caldero}, \cite[Section 12.4]{Kashiwara-crystal-notes}, \cite[Sections 5.3-5.4]{Littelmann}):}
\begin{Prop} \label{prop-main-prop-Demazure-crystal-basis} ~
\begin{enumerate}
\item For any $w \in W$, the set $\B^*_\lambda \setminus (\B_\lambda(w))^*$ is a basis for
$\ker(\pi_w)$.
\item The image (under $\pi_w$) of $(\B_\lambda(w))^*$ is a basis for $V_\lambda(w)^*$.
\item For any simple root $\alpha$, the set $(\B_\lambda(w))^* \cup \{0\}$ is invariant under
$\tilde{F}_\alpha$.
\end{enumerate}
\end{Prop}

Finally we have the following theorem about the string parametrization and Demazure modules.
As before fix a reduced decomposition $\S = (\alpha_{i_1}, \ldots, \alpha_{i_N})$ and put 
$w_k = s_{\alpha_{i_{k+1}}} \cdots s_{\alpha_{i_{N}}}$ for $k  = 0, \ldots, N$.
Take $\sigma \in V^*_\lambda \setminus \{0\}$ and let $\iota_\S(\sigma) = (a_1, \ldots, a_N)$ be its string parameters.

Similar to the definition of the string parameters,
define the $N$-tuple of integers $(e_1, \ldots, e_N)$ as follows.
Let $$ e_1 = a_1  = \max\{e \mid {F}_{\alpha_{i_1}}^e \cdot \sigma \neq 0 \}.$$
Put $\sigma_1 = {F}_{\alpha_{i_1}}^{e_1} \cdot \sigma$ and define
$$e_2 = \max\{e \mid \pi_{w_1}({F}_{\alpha_{i_2}}^e \cdot \sigma_1) \neq 0\}.$$
Put $\sigma_2 = {F}_{\alpha_{i_2}}^{e_2} \cdot \sigma_1$ and define
$$e_3 = \max\{e \mid \pi_{w_2}({F}_{\alpha_{i_3}}^e \cdot \sigma_2) \neq 0\}, \quad \textup{etc.}$$

\begin{Th}[Alternative definition of string parametrization] \label{th-alt-def-string-para}
For any $\sigma \in V^*_\lambda \setminus \{0\}$ we have $(a_1, \ldots, a_N) = (e_1, \ldots, e_N)$.
\end{Th}
We need the following lemma which can be found in \cite[Lemma 3.3.3]{Kashiwara-Demazure} { (note that we have stated the dual of the statement in \cite{Kashiwara-Demazure}).}
\begin{Lem} \label{lem-Kashiwara}
{ Let $w = s_\alpha w'$ for a simple reflection $s_\alpha$ and $\ell(w) = \ell(w') + 1$.
Let $b^* \in \B^*_\lambda$. Let us assume that for some $k \geq 0$, $\tilde{E}_\alpha^k (b^*) \in (\B_\lambda(w))^*$ and
$\tilde{F}_\alpha (b^*) = 0$. Then $b^* \in (\B_\lambda(w'))^*$.}
\end{Lem}

{
Before we start with the proof of Theorem \ref{th-alt-def-string-para} we introduce a bit of notation. For a simple root $\alpha$ and $\tau \in V_\lambda^*$ denote 
$F_\alpha^{\varphi_\alpha(\tau)} \cdot \tau$ by $L_\alpha(\tau)$, where $\varphi_\alpha$ is as in \eqref{equ-def-epsilon-phi}. 
The main point in the proof below is that if $b^* \in \B_\lambda^*$
then by repeated application of Lemma \ref{lem-Kashiwara} we see that $L_{\alpha_{i_j}} \cdots L_{\alpha_{i_1}} (b^*)$ lies in $(\B_\lambda(w_j))^*$ for every $j=1, \ldots, N$.
}

\begin{proof}[Proof of Theorem \ref{th-alt-def-string-para}]
{
Take $\sigma \in V_\lambda^* \setminus \{0\}$ and let $\iota_\S(\sigma) = (a_1, \ldots, a_N)$ be its string parameters. Also let $(e_1, \ldots, e_N)$ be as defined above.
From the definition $a_1 = e_1$. We wish to show that $a_2 = e_2$. Let $\sigma_1 = L_{\alpha_{i_1}}(\sigma)$ and write $\sigma_1 = \sum_i c_i b^*_i$ where $c_i \neq 0$ and 
$b_i^* \in \B_\lambda^*$ for all $i$. Then since $F_{\alpha_{i_1}} \cdot \sigma_1 = 0$ we conclude that 
$\tilde{F}_{\alpha_{i_1}}(b_i^*) = F_{\alpha_{i_1}} \cdot b_i^* = 0$ for all $i$.
Applying Lemma \ref{lem-Kashiwara} to the $b_i^*$, with $k = 0$, $w=w_0$ and $w' = w_1$, we see that $b_i^* \in (\B_\lambda(w_1))^*$. On the other hand by Proposition 
\ref{prop-main-prop-Demazure-crystal-basis} the set $(\B_\lambda(w_1))^* \cup \{0\}$ is stable under the $\tilde{F}_\alpha$ and the kernel of $\pi_{w_1}$ is spanned by $\B_\lambda^* \setminus (\B_\lambda(w_1))^*$.
From this it follows that $L_{\alpha_{i_2}}(\sigma_1) \notin \ker(\pi_{w_1})$. This proves that $a_2 = e_2$. Continuing the same way we get the desired result.
}
\end{proof}

Geometrically speaking, Theorem \ref{th-alt-def-string-para} states that if we regard
the elements of $V_\lambda(w_k)^*$ as sections of the $G$-linearized line bundle $L_\lambda$ restricted to the
Schubert variety $X_k$, in the step defining the string parameter $a_{k+1}$ we can restrict our 
section $\sigma_{k}$ to the Schubert variety $X_{k+1}$.

\section{Main result} \label{sec-main}
In this section we prove our main result.
Fix a lowest weight vector $\tau_\lambda$ in $H^*(G/B, L_\lambda) \cong V^*_\lambda$, i.e. a $B^-$-eigenvector.
The divisor $D_\lambda$ of $\tau_\lambda$ is $B^-$-invariant and hence does not intersect
the open opposite cell $\U^-$. In particular, $D_\lambda$ does not contain any
Schubert variety $X_w$. Let $\sigma \in H^0(G/B, L_\lambda)$ and write
$$\sigma = f_\sigma \tau_\lambda.$$ Since $\tau_\lambda$ does not vanish on
$\U^-$ then $f_\sigma$ has no pole on $\U^-$, i.e. $f_\sigma \in \c[\U^-]$.
Thus $\sigma \mapsto f_\sigma$ gives an embedding of $V_\lambda^* \cong H^0(G/B, L_\lambda)$
into $\c[\U^-] \subset \c(G/B)$.

As discussed in Section \ref{sec-Bott-Samelson} the reduced decomposition $\S$ defines a Bott-Samelson variety $\tilde{X}_\S$ and an (ordered) coordinate system $t_1, \ldots, t_N$. Recall that $v_\S$ denotes the highest term valuation associated to this coordinate system.

{ Our main result is that the valuation $v_\S$ restricts to the string parametrization $\iota_{\S}$. (We should also point out that the string parameterization $\iota_\S(\sigma)$ makes sense for any $\sigma \in V_\lambda^* \setminus \{0\}$, see Definition \ref{def-string-para-all-vectors}.)}


\begin{Th}[Main result] \label{th-main}
For any $0 \neq \sigma \in H^0(G/B, L_\lambda) \cong V_\lambda^*$ we have
$$\iota_\S(\sigma) = -v_\S(f_\sigma),$$ where $f_\sigma = \sigma / \tau_\lambda$.
\end{Th}
The negative sign appears because by definition the highest term valuation is
the negative of the highest exponent, see Example \ref{ex-Grobner-valuation}. Also note that a lowest weight vector is unique 
up to a constant and hence $v(f_\sigma)$ is independent of the choice of the lowest weight vector $\tau_\lambda$.

We then have:
\begin{Cor} \label{cor-main}
For any dominant weight $\lambda$, the string polytope $\Delta_\S(\lambda)$ can be identified with the
Newton-Okounkov body $\Delta(R(L_\lambda)) = \Delta(R(L_\lambda), v_\S, \tau_\lambda)$ associated to
the homogeneous coordinate ring $R(L_\lambda)$ and the valuation $v_\S$.
\end{Cor}
\begin{proof}
Consider the subsemigroup $\mathcal{S}_\S(\lambda) \subset \mathcal{S}_\S$ defined by
$$\mathcal{S}_\S(\lambda) = \{ (k\lambda, a) \mid a \in S_{k\lambda} \}.$$
(Recall from Section \ref{sec-string-para} that $\mathcal{S}_\S$ denotes the set of values 
of { the string parametrization} $\iota_\S$ and $S_\lambda$ is the image of $\B^*_\lambda$ under the string parametrization.)
Let $\mathcal{C}_\S(\lambda)$ be the cone generated by $\mathcal{S}_\S(\lambda)$, that is,
{ the closure of the convex hull of} $\mathcal{S}_\S(\lambda) \cup \{0\}$.
It is easily seen that the string polytope $\Delta_\S(\lambda)$ is the slice of this cone at $\lambda$, i.e.
$\Delta_\S(\lambda) = \{a \mid (\lambda, a) \in \mathcal{C}_\S(\lambda)\}$. In other words:
$$\Delta_\S(\lambda) = \overline{\conv(\bigcup_{k>0} \{ a / k \mid (k\lambda, a) \in \mathcal{S}_\S(\lambda)\} )}.$$
But by Theorem \ref{th-main}, $(k\lambda, a) \in \mathcal{S}_\S(\lambda)$ is equivalent to
$(k, -a) \in S(R(L_\lambda))$ { and hence $\Delta_\S(\lambda) = - \Delta(R(L_\lambda))$. This finishes the proof.}
\end{proof}

The rest of the section is devoted to a proof of Theorem \ref{th-main}.
The main idea in the proof is that the action of the Lie algebra of $G$ (on the
rational functions, or the sections of a line bundle) is the derivative of the action of $G$.
Hence the action of the Lie algebra elements $F_\alpha$ corresponds to the differentiation
of functions. The number of times one should apply
$F_\alpha$ to a polynomial on $\U^{-}$ to get $0$, then corresponds to
the number of times one needs to differentiate the polynomial (in an appropriate direction) 
to get $0$. In an appropriate system of coordinates,
this gives us the highest power of the first coordinate variable appearing in the polynomial.
Continuing, we get the highest term of the polynomial (with respect to a certain lexicographic order).

Consider the action of a group $G$ on a variety $X$ { (in our case the action of 
$G$ on the flag variety $X=G/B$ from the left).}
Such an action gives an action of $G$ on the field of rational functions $\c(X)$ by $(g \cdot f)(x) =
f(g^{-1} \cdot x)$, $f \in \c(X)$, and hence an action of $\Lie(G)$ on $\c(X)$.
On the other hand, every Lie algebra element $\xi \in \Lie(G)$ generates a vector field $\mathcal{V}_\xi$ on $X$. { The next lemma follows directly from the definitions.}
\begin{Lem}\label{lem-gen-vec-field}
Take $\xi \in \Lie(G)$ and $f \in \c(X)$. Then $\xi \cdot f$ is equal to
the derivative of $f$ in the direction of
the generating vector field $\mathcal{V}_{-\xi}$ on $X$, i.e.
$\xi \cdot f = df(\mathcal{V}_{-\xi})$.
In particular, $F_\alpha \cdot f$ is equal to the derivative of
$f$ in the direction of $\mathcal{V}_{-F_\alpha}$.
\end{Lem}

\begin{proof}[Proof of Theorem \ref{th-main}]
Take  $\sigma \in H^0(G/B, L_\lambda)$ and write $\sigma = f \tau_\lambda$ where $f \in \c(X)$.
We wish to show that $\iota_\S(\sigma) = - v_\S(f).$
Let $\iota_\S(\sigma) = (a_1, \ldots, a_N)$ be the string parameters of $\sigma$.
Recall (Theorem \ref{th-alt-def-string-para}) that we can alternatively define $(a_1, \ldots, a_N)$ by
\begin{eqnarray*}
a_1 &=& \max\{a \mid F_{\alpha_{i_1}}^a \cdot \sigma \neq 0 \}, \cr
a_2 &=& \max\{a \mid (F_{\alpha_{i_2}}^a F_{\alpha_{i_1}}^{a_1} \cdot \sigma)_{|X_1} \neq 0 \}, \quad \textup{etc.} \cr
\end{eqnarray*}
where $X_N \subset \cdots \subset X_0 = X$ is the sequence of Schubert varieties associated to the reduced decomposition $\S$.
As $f$ is regular on $\U^-$ it has no pole on $X_1$.
Moreover, as $\tau_\lambda$ is $U^-$-invariant, for any $\alpha$ we have:
$$F_{\alpha} \cdot \sigma = (F_{\alpha} \cdot f) \tau_\lambda.$$
This implies that: 
\begin{eqnarray*}
a_1 &=& \max \{ a \mid F_{\alpha_{i_1}}^a \cdot f \neq 0\}, \cr
a_2 &=& \max \{ a \mid (F_{\alpha_{i_2}}^a F_{\alpha_{i_1}}^{a_1} \cdot f)_{|X_1} \neq 0 \}, \quad \textup{etc.} \cr
\end{eqnarray*}

Now consider the Bott-Samelson variety $\X = \X_{\S}$ and the coordinate system $t_{1}, \ldots, t_{N}$ on the 
affine open subset $\tilde{\U} \subset \X_{\S}$ (as in Proposition \ref{prop-Bott-Samelson-coord}). { Recall the 
birational morphism $\pi_{\S}: \X \to X=G/B$}. Also recall that we have 
a sequence of Bott-Samelson varieties $\X_k$ embedded in $\X$ and 
lying over the $X_k$ such that { (1) for each $k$, $\pi_{\S}: \X_k \to X_k$ is a birational morphism}, (2) in the open set $\tilde{\U}$, 
the subvariety $\X_k$ is given by $t_1 = \cdots = t_k = 0$, and (3) $\pi_{\S}$ maps $\tilde{\U}$ to $\U^-$.

Let $\tilde{f}$ denote the pull-back of $f$ to $\X$ by $\pi_{\S}$. As $f$ is regular on $\U^-$, 
$\tilde{f}$ is regular on $\tilde{\U}$. Note that, for each $k$,
the map $\pi_{\S}$ is equivariant with respect to the actions of $F_{\alpha_{i_k}}$ on $\X_k$ and $X_k$. Thus we have:
\begin{eqnarray*}
a_{1} &=& \max \{ a \mid F_{\alpha_{i_1}}^a \cdot \tilde{f} \neq 0\}, \cr
a_2 &=& \max \{ a \mid (F_{\alpha_{i_2}}^a F_{\alpha_{i_1}}^{a_1} \cdot \tilde{f})_{|\X_1} \neq 0 \}, \quad \textup{etc.} \cr
\end{eqnarray*}

On the other hand, Proposition \ref{prop-Bott-Samelson-coord} and Lemma \ref{lem-gen-vec-field} imply that:
\begin{equation} \label{equ-F_alpha=d/dt}
F_{\alpha_{i_1}}^a \cdot \tilde{f} = (-1)^a (\partial/\partial t_1)^a \tilde{f}, 
\end{equation}
which gives us:
$$a_1 = \max \{ a \mid (\partial/\partial t_{1})^a \tilde{f} \neq 0\}.$$
Next put $\tilde{f}_1 = ((\partial/\partial t_1)^{a_1} \tilde{f})_{|\X_1}$.
By (\ref{equ-F_alpha=d/dt}), $\tilde{f}_1$ is a constant times $(F_{\alpha_{i_1}}^{a_1} f)_{|\X_1}$.
Again Proposition \ref{prop-Bott-Samelson-coord} and Lemma \ref{lem-gen-vec-field} give: 
$$a_2 = \max \{ a \mid (\partial/\partial t_2)^a \tilde{f}_1 \neq 0\}.$$
Continuing the same way, for $k=1, \ldots, N$, we have 
$$a_k = \max \{ a \mid (\partial/\partial t_{k})^a \tilde{f}_{k-1} \neq 0\},$$
where the $\tilde{f}_k$ are defined inductively by $\tilde{f}_{k} = ((\partial/\partial t_{k})^{a_{k}} \tilde{f}_{k-1})_{|\X_{k}}.$
The theorem now follows from the following elementary lemma whose proof is straightforward { but we include it here for the sake of completeness}.
\begin{Lem}\label{lem-highest-term-val-derivative}
{ Let $h \in \c[t_1, \ldots, t_N]$}. Fix the lexicographic order on the monomials
with $t_1 > \cdots > t_N$. Let $v(h) = (v_1, \ldots, v_N) \in \z_{\geq 0}^N$ be the highest exponent of $h$. We then have:
\begin{eqnarray*}
v_{1} &=& \max \{ a \mid  (\partial/\partial t_1)^a h \neq 0\}, \cr
v_2 &=& \max \{ a \mid ((\partial/\partial t_2)^a (\partial/\partial t_1)^{v_1} h) \neq 0 \}, \quad \textup{etc.}  \cr
\end{eqnarray*}
{ In other words, define $h_k$ and $a_k$ inductively by $h_0 = h$, $a_k = \max\{a \mid (\partial/ \partial t_k)^a h_{k-1} \neq 0\}$ 
and $h_k = (\partial/\partial t_k)^{a_k} h_{k-1}$ for $k=1, \ldots, N$. Then $(a_1, \ldots, a_N) = (v_1, \ldots, v_N)$.}
\end{Lem}
\begin{proof}
{ Let the highest term in $h$ be $c t_1^{v_1} \cdots t_N^{v_N}$.
Write $h(t_1, \ldots, t_N) = q_1(t_2, \ldots, t_N) t_1^{v_1} + $ (lower terms in $t_1$), where $q_1$ is a polynomial in $t_2, \ldots, t_N$. Then the highest term in $q_1$ is 
$c t_2^{v_2} \cdots t_N^{v_N}$. Since $v_1 = \max\{a \mid (\partial/ \partial t_1)^a (t_1^{v_1}) \neq 0\}$ it follows that $v_1 = a_1 = \max\{ a \mid (\partial / \partial t_1)^a h \neq 0 \}$.
Moreover $h_1 = (\partial / \partial t_1)^{v_1}(h)$ is equal to $(v_1!)q_1(t_2, \ldots, t_N)$. Write $q_1(t_2, \ldots, t_N)$ as $q_2(t_3, \ldots, t_N)t_2^{v_2} +$ (lower terms in $t_2$) and
continue as before to arrive at $(a_1, \ldots, a_N) = (v_1, \ldots, v_N)$.}
\end{proof}
\end{proof}

Corollary \ref{cor-main} is closely related to an earlier result of Okounkov on the
Gelfand-Cetlin polytopes for the symplectic group (\cite[Theorem 2]{Ok-NP}):
\begin{Rem}[Gelfand-Cetlin polytopes for symplectic group] \label{rem-Ok-SP(2n)}
Let $G = \Sp(2n, \c)$. Choose a basis $e_1, \ldots, e_{2n}$ of
$\c^{2n}$ in which the matrix of the symplectic form is
$$\left[\begin{matrix} &&&&&1 \\ &0&&&\ldots& \\ &&&1&& \\ &&-1&&& \\
&\ldots&&&0& \\ -1&&&&& \\
\end{matrix}\right].$$
Let $T$, $B^-$, $U^-$ be the subgroups of diagonal, lower triangular and
lower triangular with $1$'s on the diagonal matrices respectively.
Let $x_{ij}$, $1 \leq i,j \leq 2n$, denote the matrix entries for the elements of
$G$. Then $x_{ij}$, ${i>j; i+j \leq 2n+1}$ are coordinates for $U^-$, which can also be considered as
coordinates on the open opposite Schubert cell $\U^-$ under the map $u \mapsto uo$.
Take any dominant weight $\lambda$. As above, embed the irreducible representation
$H^0(G/B, L_\lambda) \cong V^*_\lambda$
into the polynomial ring $\c[\U^-] \cong \c[U^-]$. Then any $f \in V_\lambda^*$ can be represented as a polynomial in the
variables $x_{ij}$, ${i>j; i+j \leq 2n+1}$. Okounkov's result states that,
after a fixed linear change of coordinates in $\z^N$, the set of exponents of
the lowest terms of polynomials $f \in V_\lambda^*$ (with respect to a certain natural lexicographic order)
coincides with the set of integral points in the Gelfand-Cetlin polytope associated to $\lambda$.

We should point out that the Gelfand-Cetlin basis (for classical groups)
in general is not a crystal basis. Although as mentioned in Remark \ref{rem-string-G-C}, 
the string parameters of a crystal basis and Gelfand-Cetlin parameters of a Gelfand-Cetlin basis coincide after a fixed linear 
change of coordinates. It is interesting to note that 
in Okounkov's result the lowest term of polynomials appears as opposed to highest term valuation in our Theorem \ref{th-main}.
\end{Rem}

\begin{Rem}[Schubert varieties] \label{rem-Schubert-main}
{ Let $w$ be a Weyl group element of length $\ell$. Given a reduced decomposition 
$$\w = (\alpha_{i_1}, \ldots, \alpha_{i_\ell}), ~w = s_{\alpha_{i_1}} \cdots s_{\alpha_{i_\ell}},$$
one can define the string parameterization $\iota_\w: (\B_\lambda(w))^* \to \z_{\geq 0}^\ell$ (as in Definition \ref{def-string-para-dual-basis}) or more generally 
$\iota_\w: V_\lambda(w)^* \setminus \{0\} \to \z_{\geq 0}^\ell$ (as in Definition \ref{def-string-para-all-vectors}). 
Extend the reduced decomposition $\w$ to a reduced decompotions 
$\S = (\alpha_{i_1}, \ldots, \alpha_{i_\ell}, \ldots, \alpha_{i_N})$ for the longest element $w_0$. 
Consider:$$\mathcal{S}_\w = \bigcup_{\lambda \in \Lambda^+} \{(\lambda, \iota_\w(b^*)) \mid b^* \in (\B_\lambda(w))^*\} \subset \Lambda^+ \times \z_{\geq 0}^\ell.$$
One shows that $\mathcal{S}_\w$ coincides with $\C_\S \cap (\z_{\geq 0}^\ell \times \{0\})$ (see \cite{Littelmann}). 
One can then define the string polytope $\Delta_\w(\lambda)$ as as a slice of the string polytope $\Delta_\S(\lambda)$ by:
$$\Delta_\w(\lambda) = \Delta_\S(\lambda) \cap (\r^\ell \times \{0\}).$$
The polytope $\Delta_\w(\lambda)$ has the property that the number of integral points in it is equal to the dimension of the Demazure module $V_\lambda(w)$.
On the other hand, given a reduced decomposition $\w$, one constructs a Bott-Samelson variety $\tilde{X}_\w$ and a birational morphism $\pi_\w$ from $\tilde{X}_\w$ 
to the Schubert variety $X_w$.
Also as in Proposition \ref{prop-Bott-Samelson-coord} we have an open neighborhood $\tilde{\U}_\w \subset \tilde{X}_\w$ and a coordinate system $t_1, \ldots, t_\ell$ on 
$\tilde{\U}_\w$ which gives rise to a highest weight valuation $v_\w: \c(X_w) \setminus \{0\} \to \z^\ell$. Exactly as in the proof of Theorem \ref{th-main} one shows the following:
Take $\sigma \in H^0(X_w, {L_\lambda}_{|X_w}) \cong V_\lambda(w)^*$ and as before let $\tau_{\lambda}$ denote a lowest weight vector in $H^0(G/B, L_\lambda)$. Note that the divisor of $\tau_\lambda$ does not contain any Schubert variety and in particular ${\tau_\lambda}_{|X_w}$ is not identically zero. Thus we can write $\sigma = f_\sigma ({\tau_{\lambda}}_{|X_w})$ for some 
$f_\sigma \in \c(X_w)$. We then have:
\begin{equation} \label{equ-main-th-Schubert}
\iota_\w(\sigma) = - v_\w(f_\sigma).
\end{equation}
} 
\end{Rem}
\section{An example} \label{sec-example}
Let $G = \SL(3, \c)$. Let $\alpha, \beta, \gamma$ denote the positive roots of $G$ with $\alpha$, $\beta$ the simple roots and 
$\gamma = \alpha + \beta$. 
{ Consider the adjoint representation of $G$ on $\Lie(G)$}. It is isomorphic to the highest weight representation $V_\gamma$. 
Note that $V_\gamma \cong V_\gamma^*$. The adjoint representation $V_\gamma$ has dimension $8$ and decomposes into sum of $T$-weight spaces as: 
$$V_\gamma = W_\alpha \oplus W_\beta \oplus W_\gamma \oplus W_{-\alpha}
\oplus W_{-\beta} \oplus W_{-\gamma} \oplus W_0,$$ 
where $W_\mu$ denotes the $T$-weight space with weight $\mu$. { The weight space $W_0$ is a Cartan subalgebra in 
$\Lie(G)$ and is $2$-dimensional}. The other weight spaces have extremal weights and hence are $1$-dimensional.
One can compute that $$\mathcal{B} = \{E_\alpha, E_\beta, E_\gamma, F_\alpha, F_\beta, F_\gamma, T_1, T_2\},$$ is a crystal basis 
for $V_\gamma$ where:  
$$T_1 = \left[\begin{matrix} 1&0&0 \\ 0&1&0 \\ 0&0&-2 \\ \end{matrix}\right] \quad T_2 = \left[\begin{matrix} 
-2&0&0 \\ 0&1&0 \\ 0&0&1 \\
\end{matrix}\right].$$
The crystal graph of $V_\gamma$ is shown below:

\hspace{4.5cm}
\xygraph{ !{<0cm,0cm>;<1cm,0cm>:<0cm,1cm>::} 
!~-{@{-}@[|(2.5)]} 
!{(1,3) }*+{F_\gamma}="a" 
!{(1,6) }*+{F_\alpha}="c"
!{(1,7) }*+{0}="c0" 
!{(3,1.5) }*+{F_\beta}="b" 
!{(4,.5) }*+{0}="b0" 
!{(5,3)}*+{E_\alpha}="d" 
!{(6,2) }*+{0}="d0" 
!{(5, 6)}*+{E_\gamma}="e"
!{(6,5) }*+{0}="e0" 
!{(5,7) }*+{0}="e1" 
!{(3,7.5) }*+{E_\beta}="f"
!{(3,8.5) }*+{0}="f1" 
!{(3,5.5)}*+{T_1}="T1"
!{(4,4.5) }*+{0}="T11" 
!{(3, 3.5)}*+{T_2}="T2"
!{(3,4.5) }*+{0}="T20" 
"a":"c" 
"c":"c0"
"c":"T2"
"a":"b"
"b":@/^.5cm/"T1"
"b":"b0"
"T2":"T20"
"T2":"d"
"d":"d0"
"d":"e"
"T1":"T11"
"T1":"f"
"f":"f1"
"f":"e"
"e":"e0"
"e":"e1"
} 

\vspace{.5cm}
In the above, the slant arrows correspond to $\tilde{E}_\alpha$ and the vertical arrows correspond to $\tilde{E}_\beta$. Also $0$ denotes the 
zero vertex in the crystal graph.
 
Take the reduced decomposition $w_0 = s_\alpha s_\beta s_\alpha$. One computes that 
the string parameters of the crystal basis $\mathcal{B}$ for the choice of $\S$ are:\\

\begin{tabular}{|c|c|}  
\hline
Dual crystal basis element & String parameters\\
\hline 
$E_\alpha^*$  & $(0, 1, 0)$\\ \hline
$E_\beta^*$  & $(1, 0, 0)$\\ \hline
$E_\gamma^*$ & $(0, 0, 0)$\\ \hline
$F_\alpha^*$ & $(2, 1, 0)$\\ \hline
$F_\beta^*$ & $(0, 2, 1)$\\ \hline
$F_\gamma^*$ & $(1, 2, 1)$\\ \hline
$T_1^*$ & $(0, 1, 1)$\\ \hline
$T_2^*$ & $(1, 1, 0)$\\
\hline
\end{tabular}
\vspace{.5cm}



Let $\mathcal{U}^- = U^- o$ be the opposite open cell in $G/B$. Consider the map 
$$\Phi_{\S}: \c^3 \to U_\alpha^- \times U_\beta^- \times U_\alpha^- \to \mathcal{U}^-,$$ given by
$$\Phi_{\S}(t_1, t_2, t_3) = \exp(t_1F_\alpha)\exp(t_2F_\beta)\exp(t_3F_\alpha) o.$$ 
Identifying $\mathcal{U}^-$ with $U^-$, the map $\Phi_{\S}$ is given by:
$$\Phi_{\S}(t_1, t_2, t_3) = \left[\begin{matrix} 1& 0 & 0\\ t_1+t_3 &1& 0\\ t_2t_3 & t_2 &1 \\ \end{matrix}\right] .$$
The open opposite cell $\mathcal{U}^-$  embeds in $\p(V_\gamma^*)$ as the $U^-$ orbit of the highest weight 
vector $E_\gamma$.  
One computes that the image of $\mathcal{U}^-$ in $\p(V_\gamma^*)$ is:
$$\left[\begin{matrix} t_1t_2&-t_2&1\\ t_1t_2(t_1+t_3)&-t_2(t_1+t_3)&t_1+t_3\\t_1t_2^2t_3&-t_2^2t_3&t_2t_3\\ 
\end{matrix}\right].$$
In the coordinates $t_1, t_2, t_3$ for the opposite open cell, 
the elements of the (dual) crystal basis $\mathcal{B}^*$ for $V_\gamma^*$ correspond to polynomials in $t_1, t_2, t_3$.
We have the following list: \\

\begin{tabular}{|c|c|c|}
\hline
Dual crystal basis element & Corresponding polynomial & Exponent of the highest term \\
\hline
$E_\alpha^*$ & $-t_2$ & $(0,1,0)$\\ \hline
$E_\beta^*$ & $t_1 + t_3$ & $(1,0,0)$\\ \hline
$E_\gamma^*$ & $1$ & $(0, 0, 0)$\\ \hline
$F_\alpha^*$ & $t_1t_2(t_1+t_3)$ & $(2,1,0)$ \\ \hline
$F_\beta^*$ & $-t_2^2t_3$ & $(0,2,1)$\\ \hline
$F_\gamma^*$ & $t_1t_2^2t_3$ & $(1,2,1)$\\ \hline
$T_1^*$ & $-3t_2t_3$ & $(0,1,1)$\\ \hline
$T_2^*$ & $-3t_1t_2$ & $(1,1,0)$\\ \hline
\end{tabular}
\vspace{.5cm}

\noindent which clearly coincides with the string parameterization of the dual crystal basis for the choice of $\S$.
Below we have drawn the corresponding string polytope (Figure \ref{figure}). 

\begin{figure} 
\begin{center}
\begin{tikzpicture}[scale=.75]
\coordinate (A1) at (2.4,1.2);
\coordinate (A2) at (2.8, .4);
\coordinate (A3) at (6, 2);
\coordinate (A4) at (4.4, 1.2);
\coordinate (O) at (4, 2);
\coordinate (B1) at (6, 4);
\coordinate (B2) at (6.4, 3.2);
\coordinate (B3) at (8, 4);
\coordinate (X) at (0, 0); 
\coordinate (Y) at (10, 2);
\coordinate (Z) at (4, 6);

\fill (A1) circle [radius=3pt];
\fill (A2) circle [radius=3pt];
\fill (A3) circle [radius=3pt];
\fill (A4) circle [radius=3pt];
\fill (O) circle [radius=3pt];
\fill (B1) circle [radius=3pt];
\fill (B2) circle [radius=3pt];
\fill (B3) circle [radius=3pt];

\draw[dashed][very thick] (A1) -- (O) -- (B1);
\draw[dashed][very thick] (A3) -- (O);
\draw[very thick] (A1) -- (B1) -- (B2) -- (A2) -- cycle;
\draw[very thick] (B1) -- (B3) -- (B2);
\draw[very thick] (A2) -- (A3) -- (B3);
\draw (Y) -- (A3);
\draw (X) -- (A1);
\draw (Z) -- (O);
\end{tikzpicture}
\end{center}
\caption{String polytope for $G=SL(3, \c)$, $\lambda = \alpha + \beta$ and $w_0 = s_\alpha s_\beta s_\alpha$} \label{figure}
\end{figure}
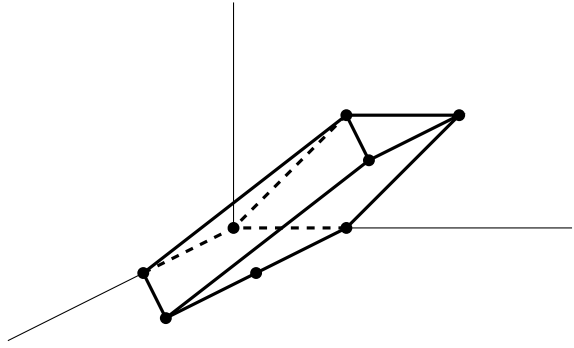

\section{A new proof of multiplicativity property of dual canonical basis} \label{sec-multi-prop}
Consider the algebra $\A = \c[G/U]$ of regular functions on { the quasi-affine homogeneous space $G/U$}.
It can be naturally identified with the algebra of $U$-invariant regular functions on $G$ for the right action 
of $U$ on $G$. Consider the action of $G \times T$ on $\A$ where $G$ acts from the left and $T$ acts
from the right (since $T$ normalizes $U$, the right action of $T$ on $G/U$ is well-defined).
One knows that as a $G \times T$-module $\A$ has a decomposition:
$$\A = \bigoplus_{\lambda \in \Lambda^+} V_\lambda^*,$$ where
$T$ acts on the irreducible representation $V_\lambda^*$ via the character
$\lambda$.

The algebra $\A$ has a basis $\B^*$  such that 
for each $\lambda$, $\B^* \cap V_\lambda^* = \B_\lambda^*$ is the dual crystal basis for $V_\lambda^*$ (see the last paragraph in Section \ref{sec-crystal}).
We call $\B^*$ the {\it dual canonical basis} for the algebra $\A$.

Consider the map $j: T \times U^- \to G/U$ given by
$(t, u) \mapsto tu U$. It identifies $T \times U^-$ with the open subset $B^- U$ in $G/U$.
The restriction map $j^*: \c[G/U] \to \c[T \times U^-]$ is then an embedding of
algebras.

{ There is a well-known partial order on the weight lattice $\Lambda$: $\lambda \geq \mu$ if $\lambda - \mu$ is 
a linear combination of the simple roots with nonnegative integer coefficients.} It has the important property that: 
{\it for $\lambda, \mu, \gamma \in \Lambda^+$, if $V_\gamma$ appears in 
$V_\lambda \otimes V_\mu$ then $\gamma \leq \lambda + \mu$.} One can extend this to a total order on $\Lambda$: 
Take a vector $\xi$ in the interior of the dual cone to the positive Weyl chamber $\Lambda_\r^+$. Moreover assume that $\xi$ is irrational with respect to 
$\Lambda$ i.e. there is no $\lambda \in \Lambda$ with $\langle \xi, \lambda \rangle = 0$. For two weights 
$\lambda, \mu \in \Lambda$ define $\lambda > \mu$ if and only if $\langle \xi, \lambda \rangle > \langle \xi, \mu \rangle$. Since $\xi$ is irrational with respect to $\Lambda$ we see that $\langle \xi, \lambda \rangle = \langle \xi, \mu \rangle$ implies that $\lambda = \mu$, i.e. 
this is a total order. Also since $\xi$ is in the interior of the dual cone to $\Lambda_\r^+$ if $\lambda$ is equal to $\mu$ plus a positive linear combination of simple roots then $\langle \xi, \lambda - \mu \rangle > 0$
and hence this ordering extends the above partial order.
Finally, equip the group $\z^N$ with the lexicographic order corresponding to the standard basis and 
define a total order on $\Lambda \times \z^N$ by: $(\lambda, \alpha) > (\mu, \beta)$ if
$\lambda > \mu$, or $\lambda = \mu$ and $\alpha > \beta$.

Given a reduced decomposition $\S$ we define a valuation
${\bf v}_\S$ on the algebra $\c[T \times U^-] \cong \c[T] \otimes \c[U^-]$
with values in $\Lambda^+ \times \z_{\geq 0}^N$ equipped with the above total order.
The image of this valuation on $\c[G/U]$ will coincide with the total image of the string parametrization, i.e.
the semigroup of all the integral points in the cone $\mathcal{C}_\S$ (see Section \ref{sec-string-para}).
Take a function $f \in \c[T \times U^-]$.
It can be written as $f = \sum_{\gamma \in \Lambda} \chi^\gamma \otimes f_\gamma$, where
$\chi^\gamma$ is the character of $T$ corresponding to a weight $\gamma$ and
$f_\gamma \in \c[U^-]$ is a polynomial on the affine space $U^-$. Let
$$\lambda = \min\{ \gamma \mid f_\gamma \neq 0 \},$$
and put $${\bf v}_\S(f) = (\lambda, v_\S(f_\lambda)),$$
where we have identified $\U^-$ and $U^-$, and $v_\S$ is the highest term valuation corresponding to the reduced decomposition 
$\S$ defined in Section \ref{sec-Bott-Samelson}.
\begin{Prop} \label{prop-v-tilde-valuation}
(1) ${\bf v}_\S$ is a valuation with one-dimensional leaves on the algebra $\c[T \times U^-]$, and hence
on $\c[G/U]$ via the embedding $j^*: \c[G/U] \to \c[T \times U^-]$.
(2) The values ${\bf v}_\S(b^*)$, for the dual canonical basis elements $b^* \in \B^*$,
are distinct.
(3) The value semigroup $S(\A, {\bf v}_\S)$
coincides with the set of integral points in the cone $\mathcal{C}_\S$.
\end{Prop}
\begin{proof}
Part (1) is a straightforward corollary of the definition of ${\bf v}_\S$ and the
fact that $v_\S$ is a valuation on $\c[U^-]$ with one-dimensional leaves. { Part (2) follows from Theorem \ref{th-main} and} the fact that, for any
dominant weight $\lambda$, $v_\S$ attains distinct values on the dual basis $\B_\lambda^*$.
{ Part (3) follows immediately from the definition of} ${\bf v}_\S$, Theorem \ref{th-main} and
Theorem \ref{th-Littelmann}.
\end{proof}

From the general properties of valuations on algebras
(Proposition \ref{prop-multi-prop}), { we now immediately recover}
the following multiplicativity property of the dual canonical basis due to
P. Caldero (\cite[Section 2]{Caldero}):

\begin{Cor}[Multiplicativity property of dual canonical basis]
\label{cor-multi-prop-canonical-basis}
Let $\lambda, \mu$ be two dominant weights. Take $b'^* \in \B_\lambda^*$ and $b''^* \in \B_\mu^*$. 
Then the product $b'^*b''^* \in V^*_{\lambda + \mu} \subset \A$ can be uniquely written as
$$b'^* b''^* = cb^* + \sum_j c_j b_j^*,$$ where
$b^*$ and the $b_j^*$ are in $\B^*_{\lambda+\mu}$, $\iota_\S(b^*) = \iota_\S(b'^*) +
\iota_\S(b''^*)$, and $\iota_\S(b_j^*) < \iota_\S(b'^*) + \iota_\S(b''^*)$ whenever $c_j \neq 0$.
\end{Cor}

{Our proof presented here is more geometric compared to the original proof in \cite{Caldero}, in the sense that it interprets the string parameterization as a geometric valuation.}

\section{SAGBI bases, valuations and toric degenerations} \label{sec-SAGBI}
This section is closely related to \cite[Section 5]{Anderson}.
Let $\c[x_1, \ldots, x_d]$ be the polynomial algebra in $d$
variables. Fix a well-ordering $<$ on $\z_{\geq 0}^d$ respecting addition, e.g. a lexicographic order.
Let $v$ denote the highest term valuation on $\c[x_1, \ldots, x_d]$ (Example \ref{ex-Grobner-valuation}), that is,
$v(f) = - \max \{ \alpha \mid c_\alpha \neq 0 \}$ where 
$f(x) = \sum_{\alpha = (a_1, \ldots, a_d) \in \z_{\geq 0}^d} c_\alpha x_1^{a_1} \cdots x_d^{a_d}$ (alternatively one can use a lowest term valuation).
Let $A$ be a subalgebra of $\c[x_1, \ldots, x_d]$. Define $S(A, v) = \{v(f) \mid f \in A \setminus \{0\}\}$. It is an additive 
semigroup in $\z_{\geq 0}^d$.
The subalgebra $A$ is said to have
a SAGBI basis ({\it Subalgebra analogue of Gr\"{o}bner basis for
Ideals}), with respect to the ordering $<$, if the semigroup $S(A, v)$ is finitely generated. A
collection of polynomials $f_1, \ldots, f_t$ such that $v(f_1),
\ldots v(f_t)$ is a set of generators for $S(A, v)$ is called a
SAGBI basis for $A$ (see \cite[Chap. 11]{Sturmfels}).

{ A remarkable property} of a SAGBI basis is that one can represent
every $h \in A$ as a polynomial in the $f_i$ in a simple
algorithmic way: write $v(h) = d_1 v(f_1) + \cdots + d_r v(f_t)$
with $d_1, \ldots, d_t \in \z_{\geq 0}$. Dividing the leading
coefficient of $h$ by the leading coefficient of ${f_1}^{d_1} \cdots
{f_t}^{d_t}$, we obtain $c$ such that the leading term of $h$ is
the same as the leading term of $c{f_1}^{d_1} \cdots {f_t}^{d_t}$.
Set $g = h - c{f_1}^{d_1} \cdots {f_t}^{d_t}$. If $g = 0$, we are
done; otherwise replace $h$ by $g$ and proceed inductively. Since
$g$ has a strictly smaller leading exponent than $h$, and $\z_{\geq 0}^d$
is well-ordered with respect to $<$, this process will
terminate, resulting in an expression for $h$ as a polynomial in the
$f_i$. This classical algorithm is referred to as the {\it subduction
algorithm}.

The SAGBI bases play an important role in computational algebra when
one deals with subalgebras of polynomials. Existence of a SAGBI
basis is a rather strong condition on the subalgebra. It is an
important unsolved problem to determine which subalgebras posses a
SAGBI basis. There are examples of subalgebras that have no SAGBI
basis with respect to any term order. On the other hand there are
subalgebras which have a SAGBI basis for one choice of a term order
and no SAGBI basis for another choice (see \cite[Chap. 11]{Sturmfels}).

The concept of a SAGBI basis can be generalized to subalgebras
of the Laurent polynomials \cite{Zinovy}. In this case, since the
set of exponents lies in $\z^d$ which is not a well-ordered set, for example with respect to any lexicographic order, one
requires that the subduction algorithm terminates in a finite number
of steps.

{Below we generalize the notion of SAGBI basis to arbitrary algebras.
Consider an algebra $A$ equipped with a valuation $v$
(with one-dimensional leaves) and with values in { a totally ordered free abelian group $\Gamma$ of finite rank.}

\begin{Def}[SAGBI basis for an arbitrary algebra and a valuation] \label{def-SAGBI-valuation}
A collection $f_1, \ldots, f_t \in A \setminus \{0\}$ is a {\it SAGBI basis}
with respect to a valuation $v$ if:
\begin{enumerate}
\item $v(f_1) \ldots, v(f_t)$ generate the semigroup $S(A, v) = \{v(f) \mid f \in A \setminus \{0\} \}$.
\item For any $h \in A \setminus \{0\}$ the subduction algorithm terminates.
\end{enumerate}
\end{Def}
Note that the existence of a SAGBI basis (i.e. termination of the subduction algorithm) in particular implies that $A$ is a 
finitely generated algebra.

In the next section we will establish the existence of SAGBI bases for rings of sections of
very ample $G$-linearized line bundles (equivalently homogeneous coordinate rings)
of flag and spherical varieties with respect to certain natural valuations. In the rest of
this section we discuss generalities on SAGBI bases and toric degenerations.

{ First we show that for a valuation $\tilde{v}$ on a ring of sections as in Section \ref{sec-Newton-Ok-bodies}, 
the condition (2) in Definition \ref{def-SAGBI-valuation}
(i.e. termination of the subduction algorithm) is automatically satisfied.
Let $X$ be a projective variety with field of rational functions $\c(X)$. Let $v: \c(X) \setminus \{0\} \to \Gamma$ be a valuation with one-dimensional leaves and with values in 
a totally ordered free abelian group $\Gamma$ of finite rank. Let $L$ be a very ample line bundle on $X$ whose ring of sections we denote as usual by $R(L)$.
Recall from Section \ref{sec-Newton-Ok-bodies} that $v$ can be extended to a valuation $\tilde{v}: R(L) \setminus \{0\} \to \z_{\geq 0} \times \Gamma$. 
Let $R$ be a graded subalgebra of $R(L)$.
}

\begin{Prop} \label{prop-SAGBI-graded-algebra}
Suppose $f_1, \ldots, f_t \in R \setminus \{0\}$ { are such that}
$\tilde{v}(f_1), \ldots, \tilde{v}(f_t)$ generate the semigroup
$$S(R, \tilde{v}) = \bigcup_{k > 0} \{(k, v(f)) \mid f \in R_k \setminus \{0\}\}.$$
Then the subduction algorithm for the $f_i$ terminates in a finite number of steps.
\end{Prop}
\begin{proof} 
{ First note that in the value semigroup $S = S(R, \tilde{v})$ every increasing sequence of elements stops. Suppose $(m_1, u_1) \leq (m_2, u_2) \leq \cdots$ is a sequence in $S$.
Then from the definition of $\leq$ on $\n \times \Gamma$ (see \eqref{equ-ordering-N-times-Z^n})  we have $m_1 \geq m_2 \geq \cdots$. Since $\n$ is well-ordered and each $S_i$ is finite we see that 
for $i$ sufficiently large we should have $(m_i, u_i) = (m_{i+1}, u_{i+1}) = \cdots$ as claimed. Now take $0 \neq h \in R$.
Write $\tilde{v}(h) = \sum_{i=1}^t k_i \tilde{v}(f_i)$. Since $v$ (and hence $\tilde{v}$) have one-dimensional
 leaves, there exists a nonzero constant $c \in \c$ such that $g_1 := h - c f_1^{k_1} \cdots f_t^{k_t}$
has the property that either $\tilde{v}(g_1) > \tilde{v}(h)$ or $g_1 = 0$. If $g_1=0$ we are done. Otherwise 
repeating the same argument for $g_1$ in place of $h$ and continuing we obtain a sequence $g_1, g_2, \ldots$ of elements of $R$ 
with $\tilde{v}(h) < \tilde{v}(g_1) < \tilde{v}(g_2) < \cdots$. As showed above an increasing sequence in $S$ should stop. Hence at some point we have
$g_i = 0$ which finishes the proof.}
\end{proof}

Let $(k, a) \in S(R, \tilde{v})$ and let $$\F_{(k, a)}= \{ f \in R \mid \tilde{v}(f) \geq (k, a) \textup{ or } f=0 \}.$$
It is straightforward to verify that the subspaces $\F_{(k, a)}$ form a decreasing filtration in $R$, i.e.
\begin{enumerate}
\item If $(k, a) < (\ell, b)$ then $\F_{(\ell,b)} \subset \F_{(k, a)}$.
\item For any $(k, a), (\ell, b) \in \z_{\geq 0} \times \Gamma$ we have
$$\F_{(k, a)} \F_{(\ell, b)} \subset \F_{(k+\ell, a+b)}.$$
\end{enumerate}
We denote the graded of the filtration $\F_{\bullet}$ by $\gr R$.

As in \cite[Proposition 5.1]{Anderson}, (see also \cite[Section 5.6]{KKh-affine}, \cite[Section 2]{AB} and \cite{Caldero}) 
one proves that the { algebra $R$ degenerates to $\gr R$:}
\begin{Th}[Degeneration of graded algebras] \label{th-graded-alg-degen}
Let $R = \bigoplus_{\geq 0} R_k$ be a graded algebra as above. Also assume that $S(R, \tilde{v})$ is
finitely generated. Then there is a finitely generated, graded, flat $\c[t]$-subalgebra $\mathcal{R} \subset R[t]$,
such that
\begin{enumerate}
\item $\mathcal{R}/t\mathcal{R} \cong  \gr R$, and
\item $\mathcal{R}[t^{-1}] \cong R[t, t^{-1}]$ as $\c[t, t^{-1}]$-algebras.
\end{enumerate}
\end{Th}

As above let us assume that $S(R, \tilde{v})$ is a finitely generated semigroup. 
One shows the following (see \cite[Section 5]{Anderson}):
\begin{Prop} \label{prop-grA-semigp-algebra}
The graded algebra $\gr R$ coincides with the semigroup algebra of the semigroup $S = S(R, \tilde{v})$.
\end{Prop}

From Theorem \ref{th-graded-alg-degen} and Proposition \ref{prop-grA-semigp-algebra} one then obtains toric degenerations. 
Let $S \subset \z_{\geq 0} \times \z^d$ be a finitely generated semigroup. Its semigroup algebra $\c[S]$ is a finitely generated $\z_{\geq 0} 
\times \z^d$ graded algebra. Thus $X_0 = \Proj(\c[S])$ is a (not necessarily normal) projective toric variety. As in Section \ref{sec-Newton-Ok-bodies} let $C(S)$ be the convex cone generated 
by $S$, and $\Delta(S) = C(S) \cap (\{1\} \times \r^d)$. Since $S$ is finitely generated $\Delta(S)$ is a rational convex polytope. The normalization of $X_0$ is the toric variety corresponding to the rational polytope $\Delta(S)$.
The toric variety $X_0$ is normal if $C(S)$ is generated by the vectors in $S_1 = S \cap (\{1\} \times \z^d)$, and moreover $S_1$ consists of all the integral points in $\Delta(S)$. In this case, the fan of $X_0$ is the normal fan of the
rational convex polytope $\Delta(S)$.  

{ We state the following corollary of Theorem \ref{th-graded-alg-degen}:}
\begin{Cor} \label{cor-toric-degen}
With notation as before, suppose the semigroup $S = S(R(L), \tilde{v})$ is finitely generated. Then the variety $X$ can be degenerated to the (not necessarily normal) toric variety $X_0$ corresponding to 
the semigroup $S$. More precisely, there is a family of irreducible varieties $\pi: \mathfrak{X} \to \c$ such that:
\begin{enumerate}
\item $\pi$ is trivial on $\c \setminus \{0\}$ with  fibre isomorphic to $X$.
\item The fibre $\pi^{-1}(0)$ is isomorphic to the toric variety $X_0$. The normalization of $X_0$ 
is the toric variety corresponding to the rational convex polytope
$\Delta(R)$.
\end{enumerate}
\end{Cor}
The above corollary is proved in \cite[Section 5]{Anderson} (and also was observed in \cite[Section 5.6]{KKh-affine}).

\section{Spherical varieties} \label{sec-spherical}
As usual let $G$ be a connected reductive algebraic group over $\c$. 
A $G$-variety is a variety equipped with an algebraic action of $G$. { A very interesting class of $G$-varieties consists of spherical varieties.}
A normal $G$-variety $X$ is called {\it spherical} if a Borel subgroup $B$ of
$G$ has a dense open orbit. If $X$ is spherical then for any $G$-linearized line bundle
$L$, the space of sections $H^0(X, L)$ is a multiplicity-free $G$-module.
Class of spherical varieties includes flag varieties and toric varieties (when $G = T$ is a torus).

Let $X$ be a projective spherical variety of dimension $d$ and $L$ a very ample
$G$-linearized line bundle on $X$. 
Generalizing the cases of toric and flag varieties, to $(X, L)$ one associates a {\it moment polytope} $\Delta_{mom}(X, L)$ as
well as a {\it string polytope} $\Delta_\S(X, L)$, lying over $\Delta_{mom}(X, L)$,
such that the degree of the line bundle $L$ is equal to $d!\Vol_d(\Delta_\S(X, L))$ (see
\cite{Ok-spherical, AB, KKh-Arnold}). The construction of $\Delta_{\S}(X, L)$ depends on the choice of a reduced decomposition $\S$.

{ In this section we show that the polytopes $\Delta_{mom}(X, L)$ and 
$\Delta_\S(X, L)$ can be realized as Newton-Okounkov bodies for $X$ with respect to certain natural choices of valuations.}
Moreover, we show that the ring of sections $R(L)$ (equivalently the homogeneous
coordinate ring of $X$) has a SAGBI basis. This then implies
the existence of toric degenerations for $X$, recovering toric degeneration results in \cite{Caldero}, 
\cite{AB} and \cite{Kaveh-Michigan}.

{ There would not be much added difficulty if instead of the ring of sections $R(L)$ 
we work with the more general setting of grade $G$-linear series, i.e. graded $G$-invariant subalgebras of $R(L)$.
Thus, for the most part, we state the definitions and results in this more general setting.}

{ As in Section \ref{sec-Newton-Ok-bodies}
let $R(L) = \bigoplus_{k \geq 0} H^0(X, L^{\otimes k})$ denote the ring of sections associated to $(X, L)$. 
 Let $R = \bigoplus_{k \geq 0} R_k$ be a graded $G$-invariant subalgebra of $R(L)$. Note that $R$ is not necessarily finitely generated. 
Examples of $R$ include in particular, rings of sections of arbitrary $G$-linearized line bundles on $X$.}
For simplicity we assume that $R_k \neq \{0\}$ for all sufficiently large $k$.
Let us write
$$R = \bigoplus_{k \geq 0} \bigoplus_{\lambda \in \Lambda^+} R_{k, \lambda},$$
where $R_{k,\lambda}$ is the $\lambda$-isotypic component of $R_k$, i.e. the
sum of all the copies of the irreducible representation $V_\lambda$ in $R_k$.
As $X$ is spherical, $R_{k, \lambda} = V_\lambda$ or $\{0\}$, for any $(k, \lambda)$.
Consider { the convex body:}
$$ \Delta_{mom}(R) = \overline{\conv(\bigcup_{k > 0}\{\lambda/k \mid R_{k,\lambda} \neq \{0\}\})}.$$
{ It is called the {\it moment convex body} of $R$ (see \cite{KKh-Arnold}).
We denote the body $\Delta_{mom}(R(L))$ associated to the whole ring of sections $R(L)$ by $\Delta_{mom}(X, L)$.

Now consider the case where $R$ is finitely generated. Recall that in this case 
the subalgebra $R^U$ of $U$-invariants is also finitely generated. Let $\{f_1, \ldots, f_s\}$ be a set of homogeneous generators for $R^U$ 
consisting of highest weight vectors.  For each $i$, let $f_i$ be of degree $m_i$ with highest weight $\lambda_i$. Then one can see that $\Delta_{mom}(R)$ is 
just the convex hull of the points $\lambda_i/m_i$ and hence is a rational polytope (see proof of Proposition \ref{prop-S-v_wt}(2)). In particular, because $L$ is very ample, $R(L)$ is finitely generated and 
thus $\Delta_{mom}(X, L)$ 
is a polytope. It is usually called the {\it moment polytope} of $(X, L)$.} It can be identified (after sending $\lambda$ 
to $\lambda^* = -w_0 \lambda$) with the image of the moment map of $X$ (regarded as a Hamiltonian space for the action of a maximal
compact subgroup of $G$) intersected with the positive Weyl chamber $\Lambda^+_\r$.

The { convex body} $\Delta_\S(R)$ is defined as:
$$\Delta_\S(R) = \bigcup_{\lambda \in \Delta_{mom}(R)}
(\lambda, \Delta_\S(\lambda)).$$ That is, the polytope $\Delta_\S(R)$ is the polytope fibered over the moment convex body $\Delta_{mom}(R)$ with
the fibre over a point $\lambda \in \Delta_{mom}(R)$ being the string polytope
$\Delta_\S(\lambda)$ (note that $\Delta_\S(\lambda)$ is defined for
any $\lambda \in \Lambda^+_\r$). 
{ In other words, if $p: \C_\S \to \Lambda^+_\r$ denotes the projection on the first factor ($\C_\S \subset \Lambda^+_\r \times \r^N$), then 
$\Delta_\S(R) = p^{-1}(\Delta_{mom}(R))$. The convex body $\Delta_\S(R)$ is called the {\it string convex body} of $R$ (\cite{KKh-Arnold}).
We denote $\Delta_\S(R(L))$ associated to the whole ring of sections by $\Delta_\S(X, L)$. 

Recall from above that if $R$ is a finitely generated algebra then $\Delta_{mom}(R)$ is a rational convex polytope. In this case, since $p: \C_\S \to \Lambda^+_\r$ is proper and $\C_\S$ is a rational convex
polyhedral cone (Theorem \ref{th-Littelmann}), we see that $\Delta_\S(R)$ is also a rational convex polytope. In particular, $\Delta_\S(X, L)$ is a polytope called the 
{\it string polytope} of $(X, L)$.}

{ Below we will see that} the convex bodies $\Delta_{mom}(R)$ and $\Delta_\S(R)$ can be
realized as Newton-Okounkov bodies for $X$ with respect to natural valuations on the ring $R(L)$.

Consider the total order on the weight lattice $\Lambda$ as in Section \ref{sec-multi-prop}, corresponding to an ordering of simple roots.
As usual extend this ordering to $\z \times \Lambda$
by defining $(k, \lambda) > (\ell, \gamma)$ if
$k < \ell$, or $k=\ell$ and $\lambda > \gamma$.
Let $f \in R$ and write $f = \sum_{(k, \gamma)} f_{k, \gamma}$, where $f_{k,\gamma} \in R_{k,\gamma}$.

\begin{Def} \label{def-weight-valuation}
Define $$\tilde{v}_{\wt}(f) = \min\{ (k, \lambda) \mid f_{k, \lambda} \neq 0\}.$$
We will refer to $\tilde{v}_\wt$ as the {\it weight valuation}.
\end{Def}

\begin{Prop} \label{Prop-weight-valuation}
The map $\tilde{v}_{\wt}$ is a valuation on $R$ with values in $\z_{\geq 0} \times \Lambda^+$
(which may not have one-dimensional leaves).
\end{Prop}
\begin{proof}
It is straightforward to check that $\tilde{v}_{\wt}$ is a pre-valuation (Definition \ref{def-pre-valuation}).
So we are required only to prove that for all $0 \neq f, g \in R(L)$,
$\tilde{v}_{\wt}(fg) = \tilde{v}_{\wt}(f) + \tilde{v}_{\wt}(g)$.
We need the following lemma.
Let $\lambda, \mu$ be dominant weights,
one knows that $V_\lambda \otimes V_\mu$ contains
$V_{\lambda+\mu}$ with multiplicity $1$.
\begin{Lem} \label{lem-pure-tensor}
The $G$-module complement of $V_{\lambda+\mu}$ in $V_\lambda \otimes V_\mu$ contains
no pure tensors, i.e. it contains no elements of the form $f \otimes g$, $f \in V_\lambda,~ g\in V_\mu$.
\end{Lem}
\begin{proof}
Let $C$ be the complement of $V_{\lambda+\mu}$ in $V_\lambda \otimes V_\mu$
and let $P$ be the set of all pure tensors in $V_\lambda \otimes V_\mu$. Both $P$ and $C$
are closed, $G$-invariant and closed under scalar multiplication and { hence the same} is true
for $P \cap C$. Thus the image of $P \cap C$ in $\p(V_\lambda \otimes V_\mu)$ is a
projective $G$-subvariety and hence contains a closed $G$-orbit which necessarily is isomorphic
to a (partial) flag variety. It follows that $P \cap C$ should contain a highest weight vector.
But the only highest weight vectors in $P$ are of the form $v_\lambda \otimes v_\mu$ which is a
highest weight vector with highest weight $\lambda + \mu$. { This shows that} $C$ contains a
copy of $V_{\lambda+\mu}$ which is not possible and hence $C \cap P = \{0\}$.
\end{proof}
Let $k, \ell \geq 0$ be integers. For dominant weights $\lambda, \mu$ let
$0 \neq f \in R_{k, \lambda}$ and $0 \neq g \in R_{\ell, \mu}$. From the definition
$\tilde{v}_{\wt}(f) = (k, \lambda)$ and $\tilde{v}(g) = (\ell, \mu)$.
It is enough to show
$\tilde{v}_{\wt}(fg) = (k+\ell, \lambda+\mu)$.
Let us write $fg = h = \sum_{(s, \gamma)} h_{\gamma}$ where $h_{\gamma} \in R_{k+\ell, \gamma}$.
{One knows that if $h_{\gamma} \neq 0$ then
$\gamma \geq \lambda + \mu$. Thus we need only to prove that $h_{\lambda+\mu} \neq 0$.}
 Since $X$ is spherical we can identify $R_{k, \lambda}$
(respectively $R_{\ell, \mu}$) with $V_{\lambda}$ (respectively $V_{\mu}$). Let
$w_{\lambda} \in V_{\lambda}$, $w_{\mu} \in V_{\mu}$ be the images of $f$, $g$ respectively. We have a commutative
diagram:
$$
\xymatrix{V_{\lambda} \times V_{\mu} \ar[d]^{\cong}
\ar[r] & V_{\lambda} \otimes V_{\mu} \ar[d]^p \\ R_{k, \lambda} \times R_{\ell, \mu} \ar[r] & R_{k+\ell}
}
$$
where the lower horizontal arrow is multiplication in $R(L)$ and
$p$ is the natural projection from the tensor product to
$R_{k, \lambda}R_{\ell, \mu} \subset R_{k+\ell}$. Now by Lemma \ref{lem-pure-tensor}
the vector $w_{\lambda} \otimes w_{\mu} \in V_{\lambda} \otimes V_{\mu}$ is not
contained in the complement of $V_{\lambda+\mu}$, and hence $p(w_\lambda \otimes w_\mu)$ is not
zero. Since the diagram is commutative, this implies that the component $h_{\lambda+\mu}$ is nonzero. The proof is finished.
\end{proof}

\begin{Prop}  \label{prop-S-v_wt} ~
\begin{enumerate}
\item The value semigroup $S(R, \tilde{v}_{\wt})$ is the weight semigroup of $R$,
that is:
$$S(R, \tilde{v}_\wt) = \{(k, \lambda) \mid R_{k, \lambda} \neq \{0\}\}.$$
{ Moreover, when $R=R(L)$ is the whole ring of sections we have:
$$S(R(L), \tilde{v}_{\wt}) = \{(k, \lambda) \mid \lambda \in k\Delta_{mom}(X, L) \cap \Lambda^+\}.$$
That is, $S(R(L), \tilde{v}_\wt)$ is the semigroup of all the
integral points in the cone over the moment polytope $\Delta_{mom}(X,L)$.}
\item If $R$ is finitely generated then the semigroup $S(R, \tilde{v}_\wt)$ { is finitely generated}.
\end{enumerate}
\end{Prop}
\begin{proof}
{ The first claim in (1) follows from the definition of $\tilde{v}_\wt$.
For the second claim note that normality of $X$ implies that $R$ and hence $R^U$ are integrally closed which in turn readily implies the claim.
To prove part (2) we observe that the semigroup $S(R, \tilde{v}_{\wt})$
coincides with the weight semigroup $S = \{(k, \lambda) \mid R^U_{k, \lambda} \neq \{0\}\}$ of the graded $T$-algebra $R^U$.
Since $R$ is assumed to be finitely generated then $R^U$ is also finitely generated. 
Let $\{f_1, \ldots, f_s\}$ be a set of generators for the $T$-algebra $R^U$ consisting of homogeneous 
weight vectors. Let $m_i$ and $\lambda_i$ be the degree and weight of the $f_i$ respectively.
It is easy to verify that the $(m_i, \lambda_i)$ generate $S$ as a semigroup which finishes the proof.
}
\end{proof}

From Proposition \ref{prop-S-v_wt} we readily obtain:
\begin{Cor} \label{cor-moment-polytope-Newton-Ok-polytope}
{ The convex body $\Delta_{mom}(R)$ coincides with the Newton-Okounkov body
$\Delta(R, \tilde{v}_{\wt})$. In particular, $\Delta_{mom}(X, L) = \Delta(R(L), \tilde{v}_{\wt})$.}
\end{Cor}

\begin{Rem}
The valuation $\tilde{v}_{\wt}$ on $R$ gives a valuation $v_{\wt}$ on $\c(X)$ with values in
$\Lambda^+$ as follows. Let $u \in \c(X)$. Since $L$ is very ample, one can find $f_1, f_2
\in H^0(X, L^{\otimes k})$, for some $k$, such that $u = f_1 / f_2$. Let
$\tilde{v}_{\wt}(f_i) = (k, \lambda_i)$, $i=1,2$. Define $v_{\wt}(u) = \lambda_1 - \lambda_2$.
As $\tilde{v}_{\wt}$ is a valuation on $R$, one verifies that $v_{\wt}$ is well-defined and
is a valuation on $\c(X)$ with values in $\Lambda$.
\end{Rem}


Finally we extend $\tilde{v}_{\wt}$ to a valuation with one-dimensional leaves. Fix a
reduced decomposition for the longest element $\S = (\alpha_{i_1}, \ldots, \alpha_{i_N})$,
$w_0 = s_{\alpha_{i_1}} \cdots s_{\alpha_{i_N}}.$ Recall that $v_\S$ denotes the highest term valuation on
the field $\c(G/B)$ constructed from a coordinate system { on the Bott-Samelson variety} associated to the reduced decomposition $\S$ 
(Section \ref{sec-Bott-Samelson}).

In each $R_{k, \lambda} \neq \{0\}$ choose a highest weight vector $h_{k, \lambda}$. The map $f \mapsto f/h_{k, \lambda}$ identifies
$R_{k, \lambda}$ with a subspace of $\c(G/B)$. By abuse of notation
let $v_\S$ denote the valuation on $R_{k, \lambda} \cong V_\lambda$ obtained by restricting $v_\S$ to the image of $R_{k, \lambda}$ 
in $\c(G/B)$. 
Note that since { the value of a valuation does not change} 
under scalar multiplication, the valuation $v_\S$ on $R_{k, \lambda}$ is independent of the choice of the highest weight vector 
$h_{k, \lambda}$.

Let $f \in R$ and write $f = \sum_{(k, \gamma)} f_{k, \gamma}$ with
$f_{k, \gamma} \in R_{k, \gamma}$. Let $\tilde{v}_{\wt}(f) = (s, \lambda)$.
\begin{Def} \label{def-tilde-v-S-valuation}
Define the valuation $\tilde{v}_\S$ on $R$ with values in $\z_{\geq 0} \times \Lambda^+ \times \z_{\geq 0}^N$ by:
$$\tilde{v}_\S(f) = (s, \lambda, v_\S(f_{s, \lambda})).$$
\end{Def}

\begin{Prop} \label{prop-tilde-v-S-valuation}
$\tilde{v}_\S$ is a valuation on $R$ with one-dimensional leaves.
\end{Prop}
\begin{proof}
That $\tilde{v}_\S$ is a valuation is straightforward. It has one-dimensional leaves because of arguments similar to those 
in Section \ref{sec-multi-prop}: For $f, g \in R$ let $\tilde{v}_\S(f) = \tilde{v}_\S(g) = (s, \lambda, a)$.
Write $f = \sum_{(k, \gamma)} f_{k, \gamma}$ and $g = \sum_{(\ell, \mu)} g_{\ell, \mu}$ with
$f_{k, \gamma} \in R_{k, \gamma}$ and $g_{\ell, \mu} \in R_{\ell, \mu}$.
Then $a = v_\S(f_{s, \lambda}) = v_\S(g_{s, \lambda})$. Since $v_\S$ has one-dimensional leaves then
there is $c \in \c$ such that $v_\S(f_{s, \lambda} - cg_{s, \lambda}) > a$ or $f_{s,\lambda} - cg_{s, \lambda} = 0$.
It then easily follows that $\tilde{v}_\S(f - cg) > (s, \lambda, a)$ which proves the proposition.
\end{proof}

We then immediately obtain the following:
\begin{Cor} \label{cor-Newton-Okounkov-polytope-spherical}
{ The convex body $\Delta_\S(R)$ coincides with
the Newton-Okounkov body $\Delta(R, \tilde{v}_\S)$. In particular, $\Delta_\S(X, L) = \Delta(R(L), \tilde{v}_\S)$.}
\end{Cor}

\begin{Cor} \label{cor-SAGBI-basis-spherical}
The semigroup $S(R(L), \tilde{v}_\S)$ is finitely generated, and hence 
the ring of sections $R(L)$ has a SAGBI basis with respect to the
valuation $\tilde{v}_\S$.
\end{Cor}
\begin{proof}
The weight semigroup $S = S(R(L), \tilde{v}_\wt)$ is finitely generated (Proposition \ref{prop-S-v_wt}) and 
the semigroup $S' = S(R(L), \tilde{v}_\S)$ projects onto the semigroup $S$. Also by Corollary \ref{cor-Newton-Okounkov-polytope-spherical} the cone of $S'$ is rational polyhedral. The claim now follows from Proposition \ref{prop-SAGBI-graded-algebra} and the next simple lemma which is a slight generalization of Gordon's Lemma.
\begin{Lem}
Let $\pi: \r^n \times \r^m \to \r^n$ denote the projection onto the first factor. Let $S \subset \z^n$ be a finitely generated semigroup
with cone $C$, i.e. $C$ is the convex hull of $S \cup \{0\}$. Let $C' \subset \r^n \times \r^m$ be a rational closed convex polyhedral cone 
which projects onto $C$ under $\pi$. 
Then the semigroup $\pi^{-1}(S) \cap C' \cap \z^{n+m}$ is finitely generated. 
\end{Lem}
\begin{proof}
Let $v_1, \ldots, v_r$ be rational generators for the cone $C' \subset \r^n \times \r^m$. Multiplying each $v_i$ with an appropriate positive integer, we can assume that $\pi(v_i) \in S$ for any $i$. Consider the set $K = \{ \sum_{i=1}^r \alpha_iv_i \mid 0 \leq \alpha_i \leq 1\}$. Clearly
$K$ is compact and hence $I = K \cap S'$ is finite. We claim that $I$ generates $S'$ as a semigroup. Note that $I$ contains $v_1, \ldots, v_r$. Take $x \in S'$. One can write $x = \sum_{i=1}^r \beta_i v_i$ with $\beta_i \in \q$ and $\beta_i \geq 0$. Then $x = x_1 + x_2$ where $x_1 = \sum_{i=1}^r [\beta_i] v_i$ and  $x_2 = \sum_{i=1}^r (\beta_i - [\beta_i])v_i$. Since the integer parts $[\beta_i]$ are all non-negative $x_1$ lies in the semigroup generated by the $v_i$, also for all $i$, $0 \leq \beta_i - [\beta_i] \leq 1$ and thus $x_2 \in I$. This shows that $x$ is in the semigroup generated by $I$ as required.
\end{proof}
\end{proof}


From the above recover the following known result (\cite{AB}, see also \cite{Kaveh-Michigan}):
\begin{Cor} \label{cor-toric-degen-spherical}
A projective spherical $G$-variety $X$ 
can be degenerated to the toric variety corresponding to the (rational) polytope $\Delta_\S(X, L)$.
\end{Cor}
\begin{proof}
{ The corollary follows directly} from Corollary \ref{cor-Newton-Okounkov-polytope-spherical}, Corollary \ref{cor-SAGBI-basis-spherical} 
and Corollary \ref{cor-toric-degen}.
\end{proof}

{ As mentioned in the introduction, our proof of Corollary \ref{cor-toric-degen-spherical} 
shows that the toric degenerations of spherical varieties/flag varieties constructed in \cite{Caldero}, \cite{AB} and \cite{Kaveh-Michigan} 
fit into the very general framework of toric degenerations associated to valuations and Newton-Okounkov bodies (\cite{Anderson}).}

\vspace{.5cm}
\noindent Kiumars Kaveh\\Department of Mathematics, University of Pittsburgh
\\Pittsburgh, PA, 15260\\{\it Email address:}
{\sf kaveh@pitt.edu}
\end{document}